\documentclass[12pt,twoside]{article}

\usepackage{amsmath,amssymb,amsthm,amscd,graphicx,color,accents}
\usepackage{enumitem}

\usepackage{natbib}

\usepackage{tikz}

\usepackage{lipsum}
 \usepackage{url}
\usepackage[hypertexnames=false,colorlinks=true,linkcolor=blue,citecolor=blue,urlcolor=blue]{hyperref}

\numberwithin{equation}{section}

\theoremstyle{plain}
\newtheorem{theorem}{Theorem}[section]
\newtheorem{proposition}[theorem]{Proposition}
\newtheorem{lemma}[theorem]{Lemma}
\newtheorem{corollary}[theorem]{Corollary}
\newtheorem{remark}[theorem]{Remark}

\newcommand\blfootnote[1]{%
  \begingroup
  \renewcommand\thefootnote{}\footnote{#1}%
  \addtocounter{footnote}{-1}%
  \endgroup
}

\def\dd{\hskip 1pt{\rm{d}}}
\def\E{\mathop{\rm E\,}\nolimits}
\def\Var{\mathop{\rm Var\,}\nolimits}
\def\Cov{{\mathop{\rm Cov}}}

\def\bR{\mathbb{R}}

\def\steq{\buildrel {\mathcal{L}} \over =}

\begin{document}

\title{\textbf{Shrinkage Estimation for the Diagonal Multivariate Natural Exponential Families}}

\author{Nikolas Siapoutis\footnote{Department of Statistics, Pennsylvania State University, University Park, PA 16802, U.S.A.  E-mail address: \href{mailto:nzs30@psu.edu}{nzs30@psu.edu}.} , Donald Richards\footnote{Department of Statistics, Pennsylvania State University, University Park, PA 16802, U.S.A. E-mail address: \href{mailto:richards@stat.psu.edu}{richards@stat.psu.edu}.} , and Bharath K.~Sriperumbudur\footnote{Department of Statistics, Pennsylvania State University, University Park, PA 16802, U.S.A.  E-mail address: \href{mailto:bks18@psu.edu}{bks18@psu.edu}.}\vspace{0.2cm} \\ 
 } 

\date{\today}

\maketitle

\vspace{-15pt}
\begin{abstract}
We study shrinkage estimation of the mean parameters of a class of multivariate distributions for which the diagonal entries of the corresponding covariance matrix are certain quadratic functions of the mean parameter. This class of distributions includes the diagonal multivariate natural exponential families.  We propose two classes of semi-parametric shrinkage estimators for the mean and construct unbiased estimators of the corresponding risk.  We establish the asymptotic consistency and convergence rates for these shrinkage estimators under squared error loss as both $n$, the sample size, and $p$, the dimension, tend to infinity. Next, we specialize these results to the diagonal multivariate natural exponential families, which have been classified as consisting of the normal, Poisson, gamma, multinomial, negative multinomial, and hybrid classes of distributions.  We establish the consistency of our estimators in the normal, gamma, and negative multinomial cases subject to the condition that $p n^{-1/3} (\log{n})^{4/3} \to 0$, and in the Poisson and multinomial cases if $p n^{-1/2} \to 0$, as $n,p \to \infty$. Simulation studies are provided to evaluate the performance of our estimators and we illustrate that, in the gamma and Poisson cases, our estimators achieve lower risk than the maximum likelihood estimator, thereby demonstrating the superiority of our estimators over the maximum likelihood estimator.


\end{abstract}

%------------------------------------------------------------------------------------------------

\section{Introduction}

Shrinkage estimators have been studied widely in statistics and have profound impact in many applications. \cite{stein1956} proved that in the case of the multivariate normal distribution on  $p$-dimensional Euclidean space, $\mathbb{R}^p$, with $p\geq 3$, there exist estimators of the population mean that dominate the sample mean under squared error loss. Subsequently, an explicit formula for such an estimator was given by \cite{james1961}. These results gave rise to {\it shrinkage estimation}, a novel approach to improved estimation of the mean, and to a broad shrinkage estimation literature for many distributions, inferential problems, and loss functions.  We refer to the monographs of \cite{MR606011}, \cite{MR1051420}, and \cite{MR0483199} for many aspects of the literature.  

In this paper, we develop shrinkage estimation for the diagonal multivariate natural exponential families of probability distributions, which were defined by \cite{bar1994diagonal}. These families include the multivariate normal, Poisson, gamma, multinomial, negative multinomial, and hybrid distributions. Later work on the diagonal multivariate natural exponential families has involved new properties and applications; see \cite{koudou1998lancaster}, \cite{bernardoff2006which}, \cite{letac2008laplace}, \cite{chatelain2006parameter},  \cite{chatelain2007bivariate}, \cite{MR3475937}. 
 
Motivated by burgeoning applications of the diagonal multivariate natural exponential families, we study simultaneous estimation of the mean parameters of random observations from those families. More broadly, we study distributions for which the diagonal entries of the covariance matrix are certain quadratic functions of the mean parameters.  We propose two classes of semi-parametric shrinkage estimators for the mean vector and, working with the squared error loss function, we also construct unbiased estimators of the corresponding risk.  

Further, we establish the asymptotic consistency and convergence rates for these shrinkage estimators under squared error loss as both $n$, the sample size, and $p$, the dimension, tend to infinity. Our results simultaneously provide shrinkage estimators for the case in which $p$ is fixed and $n \to \infty$, and also for the high-dimensional case in which both $p$ and $n$ tend to infinity.  In specializing these results to the diagonal multivariate natural exponential families, we establish consistency of our shrinkage estimators in the case of the normal, gamma and negative multinomial distributions under the condition that $p \cdot n^{-1/3}\log^{4/3}{n} \to 0$ as $n,p \to \infty$.  Further, for the Poisson and multinomial distributions, we derive consistency of our estimators subject to the condition $p \cdot n^{-1/2} \to 0$ as $n,p \to \infty$. 

The approach that we use to develop our results is motivated by the work of \cite{muandet2016kernel}, who derived shrinkage estimators which improved on the standard empirical averages estimators of kernel mean functions in reproducing kernel Hilbert spaces; \cite{xie2016optimal}, who developed shrinkage estimation for the one-dimensional natural exponential families with quadratic variance functions; and \cite{siapoutis2019}, who constructed shrinkage estimators for mean functions in general Hilbert spaces. 

In addition to our theoretical investigations, we carry out extensive simulation studies of the new shrinkage estimators and we determine that the new estimators outperform the maximum likelihood estimator. Specifically, we deduce from the simulations that, in the case of the gamma and Poisson families, our estimators achieve lower risk than the maximum likelihood estimator. In order to conduct these simulation studies, we first derive the probability density functions of these classes of distributions, a result which was unknown hitherto. Then we carry out the simulations using variable-at-a-time Metropolis algorithms for simulating observations from a multivariate gamma distribution and by means of a multivariate reduction scheme for the multivariate Poisson distribution.

The article is organized as follows. We begin by presenting in Section \ref{sec_dmnef} the basic definitions for the multivariate natural exponential families and the diagonal multivariate natural exponential families. In Section \ref{sec_shrink_to_location}, we construct a large class of semi-parametric shrinkage estimators for the mean parameter for the distributions for which the diagonal entries of the covariance matrix are certain quadratic functions of the mean parameter that shrink toward a given location and show their asymptotic properties under squared error loss. In Section \ref{sec_shrink_to_mean}, we construct another class of semi-parametric shrinkage estimators that shrinks toward the grand mean and study their properties. In Section \ref{sec_mnef}, we study the distributions that belong to the diagonal multivariate natural exponential families and establish their asymptotic consistency. Finally, simulation studies are presented in Section \ref{sec_ss}.

%------------------------------------------------------------------------------------------------

\section{The diagonal multivariate natural exponential \break families}
\label{sec_dmnef}

We provide some standard definitions for the multivariate natural exponential families. These definitions are also provided by  \cite{barndorffnielsen}, \cite{chentsov1982statistical}, \cite{morris1982natural}, \cite{brown1986fundamentals}, \cite{jvarphirgensen1987exponential}, \cite{letac1989problem},  and \cite{casalis1990families}

Suppose that $p > 1$, $\eta=(\eta_1,\dots, \eta_p) \in \mathbb{R}^p$, $x=(x_1, \dots, x_p) \in \mathbb{R}^p$, and $\mu$ is a positive measure on $\mathbb{R}^p$. Let $\langle \eta, x \rangle$ denote the inner product of $\eta$ and $x$, which is given by $\langle \eta, x \rangle=\eta^{\top} x= \sum_{j=1}^p \eta_j x_j$. Define the {\it Laplace transform of $\mu$ }as 
\begin{equation}
\label{Laplace_transform}
    L_{\mu}(\eta):=\int_{\mathbb{R}^p} \exp(\langle \eta, x \rangle) \mu(dx).
\end{equation}
Letting $H(\mu) = \{\eta \in \mathbb{R}^p: L_{\mu}(\eta)<\infty\}$, we denote by $\text{Int}( H(\mu))$ the interior of $H(\mu)$.

Let $\mathcal{M}_p$ be the set of all $\mu$ that are not concentrated on a strict affine subspace of $\mathbb{R}^p$ with $\text{Int}( H(\mu))$ being non-empty. If $\mu \in \mathcal{M}_p$ and $\eta \in \text{Int}( H(\mu))$, then 
$$
P(\eta,\mu)(dx)= (L_{\mu}(\eta))^{-1} \exp(\langle \eta, x \rangle) \mu(dx),
$$
$x \in \mathbb{R}^p$, is a probability measure. The family of distributions $F(\mu)= \{P(\eta,\mu)(dx): \eta \in \text{Int}( H(\mu))\}$ is called the {\it natural exponential family generated by} $\mu$. 

We define the {\it cumulant-generating function} (c.g.f.) of the measure $\mu$ by $k_{\mu}(\eta) = \log L_{\mu}(\eta),$
$\eta \in \text{Int}( H(\mu))$.  The {\it mean function} of the natural exponential family $F$ is
$$
m := (m_1,\dots, m_p)=\int_{\mathbb{R}^p} x P(\eta, \mu) (dx).
$$
It is a consequence of \eqref{Laplace_transform} that 
$$
m = \nabla k_{\mu}(\eta) \bigg|_{\eta=0}= \Bigg(\frac{\partial k_{\mu}(\eta)}{\partial \eta_1}, \dots, \frac{\partial k_{\mu}(\eta)}{\partial \eta_p}\Bigg) \Bigg|_{\eta=0}.
$$
Further, the {\it covariance matrix} of the natural exponential family $F$ is defined by
$$
    \Cov(m) := \int_{\mathbb{R}^p} (x-m)^\top(x-m) P(\eta, \mu) (dx).
$$
Again by \eqref{Laplace_transform}, we have 

$$
\Cov(m) = \Bigg( \frac{\partial^2 k_{\mu}(\eta)}{\partial \eta_i \partial \eta_j}: i,j = 1,\ldots, p\Bigg) \Bigg|_{\eta=0}.
$$

A natural exponential family $F$ in $\mathbb{R}^p$ is said to be {\it diagonal} if there exists functions $\alpha_j: \mathbb{R} \to \mathbb{R}$, $j=1,\ldots,p$, such that the diagonal of the matrix $\Cov(m)$ is of the form 
$$
\text{diag } \big( \Cov(m) \big)= \big(\alpha_1(m_1), \dots ,\alpha_p(m_p)\big).
$$
The family $F$ is also said to be {\it irreducible} if it is not the product of two independent natural exponential families in $\mathbb{R}^k$ and $\mathbb{R}^{p-k}$, for some $k=1,\dots, p-1$.  

For an irreducible, diagonal, natural exponential family $F$, \cite{bar1994diagonal} showed that there are only six such families in $\mathbb{R}^p$ and for these families each function $\alpha_j$ is a quadratic polynomial in $m_j$, i.e., a polynomial of degree at most two. These families are the familiar multivariate normal, Poisson, gamma, multinomial, and negative multinomial distributions, and an additional exceptional family called the hybrid distributions. 
In this article, we focus on the first five most common and well-established distributions in the literature.

%------------------------------------------------------------------------------------------------

\section{Shrinkage estimation toward a given location}
\label{sec_shrink_to_location}

Let $\Theta \subseteq \mathbb{R}$ be the space of all possible values of our parameters.  For $i=1,\ldots,n$, let $Y_i=(Y_{i1},\ldots,Y_{ip}) \in \mathbb{R}^p$ be a random vector, with distribution function $F_i$ having finite mean and covariance matrix.  We suppose that $Y_1, \ldots,Y_n$ are mutually independent random vectors, we denote the mean of each $Y_i$ by $E(Y_i) := \theta_i =(\theta_{i1},\ldots,\theta_{ip})$, where the unknown parameters $\theta_{ij} \in \Theta$ for all $i=1,\ldots,n$ and $j = 1,\ldots,p$.  Further, we denote by $\Cov(Y_i)$ the covariance matrix of $Y_i$, $i=1,\ldots,n$.  

For known constants $\nu_0, \nu_1, \nu_2 \in \mathbb{R}$, define 
\begin{equation}
\label{V_polynomial}
V(t) = \nu_0+\nu_1 t+\nu_2 t^2,
\end{equation}
$t \in \Theta$.  We assume that $\nu_0, \nu_1, \nu_2$ are such that $V(t) > 0$ for all $t \in \Theta$.  Motivated by the structure of the covariance matrices for the five most common irreducible diagonal natural exponential families, we further assume that, for each $i=1,\ldots,n$, the diagonal entries of the covariance matrix of the distribution $F_i$ are of the form 
$$
\hbox{diag}\big(\Cov(Y_i)\big) := \big(\Var(Y_{i1}),\ldots,\Var(Y_{ip})\big)
= \bigg(\frac{V(\theta_{i1})}{\tau_{i1}},\ldots,\frac{V(\theta_{ip})}{\tau_{ip}} \bigg),
$$
where the known constants $\tau_{ij}  \in \mathbb{N}$.  

In this article, we consider two classes of semi-parametric shrinkage estimators for the mean parameters $\theta_i$, $i=1,\ldots, n$.  One class of estimators will shrink $Y_{ij}$ toward a given location $\mu_j \in \mathbb{R}$, $j=1,\ldots, p$, and the second class of estimators will provide shrinkage toward the mean vector $\bar{Y}_j = n^{-1} \sum_{i=1}^{n} Y_{ij}$, for $j=1,\ldots, p$.  

Let $b=(b_1, \ldots ,b_n)$ where $b_i \in [0,1]$, $i=1,\ldots,n$.  Also, let $\mu=(\mu_1,\ldots,\mu_p) \in \mathbb{R}^p$.  In this section, we consider shrinkage estimators of the form
\begin{equation}
\label{estimator_mu}
\hat{\theta}_{ij}^{b,\mu}=(1-b_i)Y_{ij}+b_i\mu_j,   
\end{equation}
$i=1,\ldots,n, \; j=1,\ldots p$, that shrink $Y_{ij}$ toward a given location $\mu_j \in \mathbb{R}$. 
We also require that $|\mu_j|\leq \max\{|Y_{il}|: i= 1,\ldots,n, l=1,\ldots,p\}$ for all $j=1,\ldots,p$, so that shrinkage will take place toward a vector $\mu_j$ that is within the range of the data. 

The squared error loss of the estimators in \eqref{estimator_mu} is
\begin{align}
\label{loss_1}
\ell_{b,\mu}= \frac{1}{np} \sum_{i=1}^{n} \sum_{j=1}^{p} (\hat{\theta}_{ij}^{b,\mu}-\theta_{ij})^2,
\end{align}
and we define its risk as the expected value of the loss function given by
\begin{align}
\label{risk_1}
R_{b,\mu}=E(\ell_{b,\mu}).
\end{align}

We wish to find an optimal choice of $b$ and $\mu$ that minimizes the risk \eqref{risk_1}. However, this is not feasible since the risk depends on the unknown parameters $\theta_{ij}$. Similar to \cite{xie2016optimal}, we propose and minimize an unbiased estimator of its risk, given in Proposition \ref{proposition_unbiased}.

\begin{proposition}
\label{proposition_unbiased}
An unbiased estimator of the risk $R_{b,\mu}$ in \eqref{risk_1} is given by
\begin{align}
\label{URE}
\hat{R}_{b,\mu}= \frac{1}{np} \sum_{i=1}^{n} \sum_{j=1}^{p} \bigg[b_i^2(Y_{ij} - \mu_j)^2 + (1-2b_i)\frac{V(Y_{ij})}{\tau_{ij}+\nu_2}\bigg].
\end{align}
\end{proposition}

\begin{proof} 
Taking the expectation of the estimator in \eqref{URE}, we obtain 
\begin{align}
\label{proposition_proof_1}
E\big(\hat{R}_{b,\mu}\big)&= \frac{1}{np} \sum_{i=1}^{n} \sum_{j=1}^{p} \bigg[b_i^2E(Y_{ij}- \mu_j)^2 + (1-2b_i)E\bigg(\frac{V(Y_{ij})}{\tau_{ij}+\nu_2}\bigg)\bigg] .
\end{align}
The term $E(Y_{ij}- \mu_j)^2$ in \eqref{proposition_proof_1} can be expressed as
\begin{align*}
E(Y_{ij}- \mu_j)^2&=\Var(Y_{ij}-\mu_j)+ \big( E(Y_{ij} -\mu_j) \big)^2\\
&=\Var(Y_{ij})+ \big( E(Y_{ij}) -\mu_j \big)^2. 
\end{align*}
Further, the term $E \big(V(Y_{ij})/(\tau_{ij}+v_2)\big)$ in \eqref{proposition_proof_1} can be simplified to
\begin{align}
\label{proposition_proof_0}
E\bigg(\frac{V(Y_{ij})}{\tau_{ij}+\nu_2}\bigg)&=\frac{\nu_0+\nu_1E(Y_{ij})+\nu_2E(Y_{ij}^2)}{\tau_{ij}+\nu_2} \nonumber \\
&=\frac{\nu_0+\nu_1E(Y_{ij})+\nu_2\Var(Y_{ij})+\nu_2(E(Y_{ij}))^2}{\tau_{ij}+\nu_2} \nonumber \\
&=\frac{\tau_{ij}\Var(Y_{ij})+\nu_2\Var(Y_{ij})}{\tau_{ij}+\nu_2} \nonumber \\
&=\Var(Y_{ij}),
\end{align}
where we use the fact that 
$$
\tau_{ij} \Var(Y_{ij}) = V(\theta_{ij}) = \nu_0+\nu_1E(Y_{ij})+\nu_2(E(Y_{ij}))^2 .
$$
On summing over all $i, j$, we find that \eqref{proposition_proof_1} equals
\begin{align}
\label{proposition_proof_2}
E\big(\hat{R}_{b,\mu}\big)&= \frac{1}{np} \sum_{i=1}^{n} \sum_{j=1}^{p} \big[b_i^2 \big( \Var(Y_{ij})+(E(Y_{ij})- \mu_j)^2 \big) + (1-2b_i)\Var(Y_{ij})\big] \nonumber \\
&= \frac{1}{np} \sum_{i=1}^{n} \sum_{j=1}^{p} \big[b_i^2(\theta_{ij}-\mu_j)^2+(1-b_i)^2\Var(Y_{ij})\big].
\end{align}
Now, note that
\begin{align}
\label{risk_proof_1}
R_{b,\mu} &= \frac{1}{np} \sum_{i=1}^{n} \sum_{j=1}^{p} E(\hat{\theta}_{ij}^{b,\mu}-\theta_{ij})^2 \nonumber \\
&= \frac{1}{np} \sum_{i=1}^{n} \sum_{j=1}^{p} E\big( (1-b_i)Y_{ij}+b_i \mu_j-\theta_{ij}\big)^2 \nonumber \\
&= \frac{1}{np} \sum_{i=1}^{n} \sum_{j=1}^{p} E\big( (1-b_i)Y_{ij}+b_i\mu_j-\theta_{ij} +b_i\theta_{ij}-b_i\theta_{ij}\big)^2.
\end{align}
By rearranging the terms and applying the trinomial identity $(x+y+z)^2\equiv x^2 + y^2 + z^2 + 2xy + 2yz + 2zx$ for $x,y,z \in \mathbb{R}$, we obtain \eqref{risk_proof_1} in the form
\begin{align}
\label{risk_proof_2}
R_{b,\mu}&= \frac{1}{np} \sum_{i=1}^{n} \sum_{j=1}^{p} \big[ b_i^2(\theta_{ij}-\mu_j)^2+(1-b_i)^2E(Y_{ij}^2)+(1-b_i)^2\theta_{ij}^2-2(1-b_i)^2\theta_{ij}^2 \nonumber \\
&  \qquad \qquad \qquad \qquad -2b_i(1-b_i)\theta_{ij}(\theta_{ij}-\mu_j)+2b_i(1-b_i)\theta_{ij}(\theta_{ij}-\mu_j)\big] \nonumber \\
&=\frac{1}{np} \sum_{i=1}^{n} \sum_{j=1}^{p} \big[b_i^2(\theta_{ij}-\mu_j)^2+(1-b_i)^2\Var(Y_{ij})\big].
\end{align}
It follows from \eqref{proposition_proof_2} and \eqref{risk_proof_2} that $E(\hat{R}_{b,\mu})=R_{b,\mu}$ and therefore the $\hat{R}_{b,\mu}$ estimator in \eqref{URE} is an unbiased estimator of the risk in \eqref{risk_1}. 
\end{proof}

Consequently, we aim to find $\hat{b}^{*}$ and $\hat{\mu}^{*}$ that minimize \eqref{URE} over the set 
\begin{align}
\label{min_set}
\Lambda=\Big\{(b,\mu):& b_i \in [0,1], \nonumber \\%\; b_i \leq b_k \text{ for any $i,k= 1, \ldots,n $ whenever }  \tau_{i \boldsymbol{\cdot}} \geq \tau_{k \boldsymbol{\cdot}},  \\ & \quad \quad  \quad  \text{ and } 
& |\mu_j|\leq \max\{|Y_{il}|: i= 1,\ldots,n, l=1,\ldots,p\} \text{ for $j=1,\ldots, p$}\Big\},
\end{align}
and the shrinkage estimator in \eqref{estimator_mu} becomes
\begin{align}
\label{estimator_mu_URE}
\hat{\theta}_{i}^{\hat{b}^*,\hat{\mu}^*}=(1-\hat{b}_{i}^{*})Y_i+\hat{b}_i^{*} \hat{\mu}^{*},
\end{align}
$i=1,\ldots,n$.

Naturally, the next issue that arises is to determine whether or not $\hat{b}^{*}$ and $\hat{\mu}^{*}$ are good estimators of the actual parameters. In other words, if the estimator $\hat{R}_{b,\mu}$ is a good approximation of the risk $R_{b,\mu}$ then we will expect to have correspondingly good results regarding estimation of $b$ and $\mu$. In Theorem \ref{uniform_URE} we will show that not only is the estimator in \eqref{URE} a good estimator of the risk in \eqref{risk_1}, but also it is uniformly close to the loss $\ell_{b, \mu}$.  Further, in Theorem \ref{theorem_optimality}, we will provide the asymptotic behavior of our proposed estimators \eqref{estimator_mu_URE} among a large class of shrinkage estimators.

Now, we introduce the assumptions that we need for proving the main results of the article:
\begin{enumerate}[label=(\Alph*)]
\item \label {A} $ \limsup\limits_{n, p \rightarrow \infty} \sum_{i=1}^{n} \sum_{j=1}^{p} \Var(Y_{ij})/np<\infty$,
\item \label {B} $ \limsup\limits_{n, p \rightarrow \infty} \sum_{i=1}^{n}\sum_{j=1}^{p}\Var(Y_{ij})\theta_{ij}^2/np<\infty$,
\item \label {C} $ \limsup\limits_{n, p \rightarrow \infty} \sum_{i=1}^{n}\sum_{j=1}^{p} \Var(Y_{ij}^2)/np<\infty$,
\item \label {D} $\sup\limits_{i,j} \big(\tau_{ij}/(\tau_{ij}+\nu_2)\big)^2<\infty$,
\item \label {E}  $E(\max\limits_{i,j} Y_{ij}^2)/n= \text{O}(n^{-\alpha}p^{\beta})$ for some $\alpha>0$ and $\beta\ge 0$ such that $n^{-\alpha}p^\beta\rightarrow 0$ as $n,p\rightarrow\infty$.
\end{enumerate}

We now present the two main results on the first class of shrinkage estimators given in \eqref{estimator_mu_URE}.  The first result establishes the uniform convergence of the $\hat{R}_{b,\mu}$ estimator to the actual loss.

\begin{theorem}
\label{uniform_URE}
Suppose that the assumptions \textnormal{\ref{A}-\ref{E}} hold.  Then, as $n, p \rightarrow\infty$, 
\begin{align}
\label{uniform_URE_result}
E\Big(\sup_{(b,\mu) \in \Lambda} |\hat{R}_{b,\mu}-\ell_{b,\mu}|\Big) = \textnormal{O}(n^{-1/2}+n^{-\alpha/2}p^{\beta/2}).
\end{align}
\end{theorem}

The second theorem shows that our proposed estimator is asymptotically optimal among a large class of shrinkage estimators.

\begin{theorem}
\label{theorem_optimality}
Suppose that the assumptions \textnormal{\ref{A}-\ref{E}} hold, and consider any shrinkage estimator of the form
\begin{align}
\hat{\theta}_{i}^{\hat{b},\hat{\mu}}=(1-\hat{b}_{i})Y_i+\hat{b}_i \hat{\mu}, \nonumber
\end{align}
 $i=1,\ldots, n,$ where $(\hat{b},\hat{\mu}) \in \Lambda$. Then, as $n, p \rightarrow \infty$, 
\begin{align}
\label{optimaility_result_1}
 \ell_{\hat{b}^{*},\hat{\mu}^{*} } \leq \ell_{\hat{b},\hat{\mu}} + \textnormal{O}_{\mathbb{P}}(n^{-1/2}+n^{-\alpha/2}p^{\beta/2}), 
\end{align}
and
\begin{align}
\label{optimaility_result_2}
\limsup\limits_{n, p\rightarrow\infty} (R_{\hat{b}^{*},\hat{\mu}^{*} } - R_{\hat{b},\hat{\mu}})\leq 0.
\end{align}
\end{theorem}

Because the proofs of Theorem \ref{uniform_URE} and \ref{theorem_optimality} are lengthy, we will present them at the end of this section.

\begin{remark}
{\rm
Interestingly, assumptions \ref{A} and \ref{B} hold if the condition
\begin{align}
\limsup\limits_{n\rightarrow \infty}  \frac{1}{np}\sum_{i=1}^{n}\sum_{j=1}^{p} \frac{\theta_{ij}^k}{\tau_{ij}}<\infty,\nonumber
\end{align}
$k=0,1, 2, 3, 4$ is satisfied. The above statement holds because
\begin{align*}
 \sum_{i=1}^{n}\sum_{j=1}^{p}\Var(Y_{ij}) &=
\sum_{i=1}^{n}\sum_{j=1}^{p} \frac{\nu_0+\nu_1\theta_{ij}+\nu_2 \theta_{ij}^2}{\tau_{ij}} \\
&=\nu_0\sum_{i=1}^{n}\sum_{j=1}^{p}\frac{1}{\tau_{ij}}+\nu_1\sum_{i=1}^{n}\sum_{j=1}^{p}\frac{\theta_{ij}}{\tau_{ij}}+\nu_2\sum_{i=1}^{n}\sum_{j=1}^{p} \frac{\theta_{ij}^2}{\tau_{ij}}
\end{align*}
and 
\begin{align*}
\sum_{i=1}^{n} \sum_{j=1}^{p}\Var(Y_{ij})\theta_{ij}^2 &=
\nu_0\sum_{i=1}^{n}\sum_{j=1}^{p} \frac{\theta_{ij}^2}{\tau_{ij}}+\nu_1\sum_{i=1}^{n}\sum_{j=1}^{p}\ \frac{\theta_{ij}^3}{\tau_{ij}}+\nu_2\sum_{i=1}^{n}\sum_{j=1}^{p} \frac{\theta_{ij}^4}{\tau_{ij}}.
\end{align*}
Thus, assumptions \ref{A} and \ref{B}  place restrictions on the growth of $\theta_{ij}^k / \tau_{ij}$, $k=0,\ldots,4,$ for all $i,j$. 
}
\end{remark}

\begin{remark}
{\rm
On the other hand, assumption \ref{C}, because it involves the fourth moment of each $Y_{ij}$, places restrictions on the growth of the kurtosis of $Y_{ij}$. This leads us to observe that the condition, 
\begin{align}
\label{sufficient_condition_1}
\limsup\limits_{n\rightarrow \infty} \frac{1}{np} \sum_{i=1}^{n}\sum_{j=1}^{p} E(Y_{ij}^4)<\infty,
\end{align}
$i=1,\ldots,p$ implies conditions \ref{A}, \ref{B}, and \ref{C}.  We prove this as follows.

By Jensen's inequality, we have $\Var(Y_{ij})\leq E(Y_{ij}^2)\leq ( E(Y_{ij}^4))^{1/2}$. Applying the Cauchy-Schwarz inequality, we obtain 
\begin{align*}
\frac{1}{np} \sum_{i=1}^{n}\sum_{j=1}^{p}\Var(Y_{ij}) 
&\leq \frac{1}{np} \sum_{i=1}^{n}\sum_{j=1}^{p} \big( E(Y_{ij}^4)\big)^{1/2} 
\leq \bigg( \frac{1}{np} \sum_{i=1}^{n}\sum_{j=1}^{p} E(Y_{ij}^4) \bigg)^{1/2},   
\end{align*}
and therefore \eqref{sufficient_condition_1} implies \ref{A}.

As for \ref{B}, we have
\begin{align*}
\frac{1}{np} \sum_{i=1}^{n}\sum_{j=1}^{p} \Var(Y_{ij})\theta_{ij}^2 
& \leq \frac{1}{np} \sum_{i=1}^{n} \sum_{j=1}^{p}E(Y_{ij}^2)\theta_{ij}^2 \\
&\leq \bigg(\frac{1}{np} \sum_{i=1}^{n}\sum_{j=1}^{p} E(Y_{ij}^4)\bigg)^{1/2} \bigg( \frac{1}{np} \sum_{i=1}^{n}\sum_{j=1}^{p}\theta_{ij}^4\bigg)^{1/2},
\end{align*}
where the second inequality follows from the Cauchy-Schwarz inequality. By Jensen's inequality, as $(E(Y_{ij}^2))^2 \leq E(Y_{ij}^4)$, then we see that \eqref{sufficient_condition_1} implies \ref{B}.

Finally, since $\Var(Y_{ij}^2)\leq E(Y_{ij}^4)$ then
\begin{align*}
\frac{1}{np} \sum_{i=1}^{n} \sum_{j=1}^{p}\Var(Y_{ij}^2)
&\leq \frac{1}{np} \sum_{i=1}^{n}\sum_{j=1}^{p}E(Y_{ij}^4),
\end{align*}
and therefore \eqref{sufficient_condition_1} implies \ref{C}.
}
\end{remark}

\begin{remark}
{\rm
Later in Lemma \ref{sufficient_conditions_E}, we will derive sufficient conditions for which condition \ref{E} holds in the special case where each $Y_{ij}$ belongs to a natural exponential family with quadratic variance functions.
}
\end{remark}

Let us define the quantities 
\begin{align*}
&T_1= \frac{1}{np} \bigg| \sum_{i=1}^{n} \sum_{j=1}^{p} 
\bigg[ \frac{V(Y_{ij})}{\tau_{ij}+v_2}-(Y_{ij}-\theta_{ij})^2 \bigg] \bigg| \\
&T_2=\frac{2}{np} \sup_{(b,\mu) \in \Lambda}  \bigg| \sum_{i=1}^{n} \sum_{j=1}^{p} 
b_i \bigg[ \frac{V(Y_{ij})}{\tau_{ij}+v_2}-(Y_{ij}-\theta_{ij})^2 \bigg] \bigg|, \\
\intertext{and}
&T_3=\frac{2}{np} \sup_{(b,\mu) \in \Lambda}  \bigg| \sum_{i=1}^{n} \sum_{j=1}^{p} 
b_i(Y_{ij}-\theta_{ij})(\theta_{ij}-\mu_{j}) \bigg|.
\end{align*}
In proving the previously-stated theorems, we need the following lemma which provides the convergence properties of $T_1$, $T_2$, and $T_3$.

\begin{lemma}
\label{lemma_T1_T2_T3}
Suppose that the assumptions \textnormal{\ref{A}-\ref{E}} hold.  Then, $E(|T_1|) = \textnormal{O}(n^{-1/2})$, $E(|T_2|) = \textnormal{O}(n^{-1/2})$, and $E(|T_3|)= \textnormal{O}(n^{-1/2} + n^{-\alpha/2}p^{\beta/2})$ as $n, p \rightarrow\infty$.
\end{lemma}

Since the proof of Lemma \ref{lemma_T1_T2_T3} is lengthy, we provide it in Appendix \ref{appendix_A}.  Now, we have all the results needed to prove Theorems \ref{uniform_URE} and \ref{theorem_optimality}.

\medskip

\begin{proof}[Proof of Theorem \ref{uniform_URE}]
By applying \eqref{estimator_mu}, \eqref{loss_1}, and \eqref{URE}, we obtain 
\begin{align}
\label{uniform_URE_proof_1}
\hat{R}_{b,\mu}-\ell_{b,\mu}&=\frac{1}{np} \sum_{i=1}^{n} \sum_{j=1}^{p} \bigg[b_i^2(Y_{ij} - \mu_j)^2 + (1-2b_i)\frac{V(Y_{ij})}{\tau_{ij}+\nu_2}\bigg] \nonumber \\
 & \qquad - \frac{1}{np} \sum_{i=1}^{n} \sum_{j=1}^{p} \big[(1-b_i)Y_{ij}+b_i\mu_j-\theta_{ij}\big]^2. 
\end{align}
Expanding the second term in \eqref{uniform_URE_proof_1}, we have
\begin{align}
\label{uniform_URE_proof_2}
& \frac{1}{np} \sum_{i=1}^{n} \sum_{j=1}^{p} \big[(1-b_i)Y_{ij}+b_i\mu_j-\theta_{ij}\big]^2 \nonumber \\
& \ \ = \frac{1}{np} \sum_{i=1}^{n} \sum_{j=1}^{p} \big[(Y_{ij}-\theta_{ij})^2+b_i^2(Y_{ij}-\mu_j)^2 
-2b_i(Y_{ij}-\theta_{ij})(Y_{ij}-\mu_j)\big].
\end{align}
Rearranging terms in \eqref{uniform_URE_proof_1} and applying  \eqref{uniform_URE_proof_2}, we find that   \eqref{uniform_URE_proof_1} is given by
\begin{align}
\label{uniform_URE_proof_3}
\hat{R}_{b,\mu}-\ell_{b,\mu}&=\frac{1}{np} \sum_{i=1}^{n} \sum_{j=1}^{p} (1-2b_i)\frac{V(Y_{ij})}{\tau_{ij}+\nu_2} -\frac{1}{np} \sum_{i=1}^{n} \sum_{j=1}^{p}(Y_{ij}-\theta_{ij})^2  \nonumber \\
 & \qquad + \frac{1}{np} \sum_{i=1}^{n} \sum_{j=1}^{p} 2b_i(Y_{ij}-\theta_{ij})(Y_{ij}-\mu_j).
\end{align}
For the third term of \eqref{uniform_URE_proof_3}, we note
\begin{align*}
\frac{1}{np} & \sum_{i=1}^{n} \sum_{j=1}^{p} 2b_i(Y_{ij} - \theta_{ij})(Y_{ij}-\mu_j) \\ 
&= \frac{1}{np} \sum_{i=1}^{n} \sum_{j=1}^{p} 2b_i(Y_{ij}-\theta_{ij})(\theta_{ij}-\mu_j+Y_{ij}-\theta_{ij}) \\
&=\frac{1}{np} \sum_{i=1}^{n} \sum_{j=1}^{p} 2b_i(Y_{ij}-\theta_{ij})(\theta_{ij}-\mu_j) + \frac{1}{np} \sum_{i=1}^{n} \sum_{j=1}^{p} 2b_i(Y_{ij}-\theta_{ij})^2.
\end{align*}
Hence, we find that \eqref{uniform_URE_proof_3} equals
\begin{align}
\label{uniform_URE_proof_5}
\hat{R}_{b,\mu}-\ell_{b,\mu} &= \frac{1}{np} \sum_{i=1}^{n} \sum_{j=1}^{p} 
(1-2b_i)\bigg[ \frac{V(Y_{ij})}{\tau_{ij}+\nu_2}-(Y_{ij}-\theta_{ij})^2 \bigg] \nonumber \\
& \qquad\qquad\qquad + \frac{2}{np} \sum_{i=1}^{n} \sum_{j=1}^{p} b_i(Y_{ij}-\theta_{ij})(\theta_{ij}-\mu_{j}). 
\end{align}
Taking the supremum over $\Lambda$ of the absolute value of the eq. \eqref{uniform_URE_proof_5}, we have
\begin{align}
\label{uniform_URE_proof_6}
\sup_{(b,\mu) \in \Lambda} | \hat{R}_{b,\mu}-\ell_{b,\mu}|
&\leq T_1+T_2+T_3.
\end{align}
By taking expectations on both sides of \eqref{uniform_URE_proof_6} and applying Lemma \ref{lemma_T1_T2_T3}, we find that \eqref{uniform_URE_result} holds. 
\end{proof}

\begin{proof}[Proof of Theorem \ref{theorem_optimality}]
Since $(\hat{b}^{*},\hat{\mu}^{*})$ is a minimizer of the estimator \eqref{URE}, we have that $\hat{R}_{\hat{b}^{*},\hat{\mu}^{*}}\leq \hat{R}_{\hat{b}, \hat{\mu}}$ for any $\hat{b} \text{ and }\hat{\mu}$.  For $\epsilon > 0$, 
\begin{align}
\label{optimality_theorem_proof_0}
P\big(\ell_{\hat{b}^{*},\hat{\mu}^{*} } \geq \ell_{\hat{b},\hat{\mu}} + \epsilon\big) &\leq 
 P\Big(\ell_{\hat{b}^{*},\hat{\mu}^{*}}- \hat{R}_{\hat{b}^{*},\hat{\mu}^{*}}  \geq \ell_{\hat{b},\hat{\mu}} -\hat{R}_{\hat{b}, \hat{\mu}} + \epsilon \Big) \nonumber \\
&\leq P\Big( | \ell_{\hat{b}^{*},\hat{\mu}^{*}}- \hat{R}_{\hat{b}^{*},\hat{\mu}^{*}}|  \geq \epsilon /2 \Big) 
+ P\Big(| \ell_{\hat{b},\hat{\mu}} - \hat{R}_{\hat{b}, \hat{\mu}}| \geq \epsilon/2 \Big). 
\end{align}

Now, since 
$$| \ell_{\hat{b}^{*},\hat{\mu}^{*}}- \hat{R}_{\hat{b}^{*},\hat{\mu}^{*}}| \leq \sup\limits_{(b,\mu) \in \Lambda} |\ell_{b,\mu}-\hat{R}_{b, \mu} |$$
and 
$$| \ell_{\hat{b},\hat{\mu}}-\hat{R}_{\hat{b}, \hat{\mu}}| \leq \sup\limits_{(b,\mu) \in \Lambda} |\ell_{b,\mu}-\hat{R}_{b, \mu} |,$$
then it follows that \eqref{optimality_theorem_proof_0} is bounded above by 
\begin{align}
2P\Big(\sup\limits_{(b,\mu) \in \Lambda} |\ell_{b,\mu}-\hat{R}_{b, \mu} | \geq \epsilon/2 \Big).\nonumber
\end{align}

By Markov's inequality, we obtain 
\begin{align}
\label{optimality_theorem_proof_12}
P\big(\ell_{\hat{b}^{*},\hat{\mu}^{*} } \geq \ell_{\hat{b},\hat{\mu}} + \epsilon\big) &\leq \frac{4}{\epsilon} E \Big(\sup\limits_{(b,\mu) \in \Lambda} |\ell_{b,\mu}-\hat{R}_{b, \mu} | \Big).
\end{align}
By applying Theorem \ref{uniform_URE} to \eqref{optimality_theorem_proof_12}, we obtain \eqref{optimaility_result_1}.

To prove \eqref{optimaility_result_2}, we have
\begin{align}
\label{optimality_theorem_proof_2}
\ell_{\hat{b}^{*},\hat{\mu}^{*} } - \ell_{\hat{b},\hat{\mu}} &= \big(\ell_{\hat{b}^{*},\hat{\mu}^{*} } - \hat{R}_{\hat{b}^{*},\hat{\mu}^{*}}\big) + \big(\hat{R}_{\hat{b}^{*},\hat{\mu}^{*}}  - \hat{R}_{\hat{b}, \hat{\mu}}\big) + \big( \hat{R}_{\hat{b}, \hat{\mu}}-\ell_{\hat{b},\hat{\mu}} \big), 
\end{align}
where we add and subtract the terms $\hat{R}_{\hat{b}^{*},\hat{\mu}^{*}}$ and $\hat{R}_{\hat{b}, \hat{\mu}}$. Again, using the fact that $\hat{R}_{\hat{b}^{*},\hat{\mu}^{*}}  \leq  \hat{R}_{\hat{b}, \hat{\mu}}$ for any $\hat{b} \text{ and }\hat{\mu}$, we find that \eqref{optimality_theorem_proof_2} is bounded above by
\begin{align*}
\big(\ell_{\hat{b}^{*},\hat{\mu}^{*} } -\hat{R}_{\hat{b}^{*},\hat{\mu}^{*}}\big) + \big(\hat{R}_{\hat{b}, \hat{\mu}}-\ell_{\hat{b},\hat{\mu}} \big), 
\end{align*}
Hence, we obtain 
\begin{align}
\label{optimality_theorem_proof_3}
\ell_{\hat{b}^{*},\hat{\mu}^{*} } - \ell_{\hat{b},\hat{\mu}}  &\leq 2\sup_{(b,\mu) \in \Lambda} |\hat{R}_{b, \mu}- \ell_{b,\mu}|,
\end{align}
By taking expectations in \eqref{optimality_theorem_proof_3} and applying Theorem \ref{uniform_URE}, we obtain the desired result.
\end{proof}

%----------------------------------------------------------------------------------------%

\section{Shrinkage estimation toward the grand mean}
\label{sec_shrink_to_mean}

In the previous section, we consider the class of semi-parametric shrinkage estimators of the mean parameters $\theta_i$, $i=1,\ldots, n$, given by \eqref{estimator_mu}. That class of estimators shrinks each $Y_{ij}$ toward a location $\mu_j$ that is determined by solving an optimization problem; specifically, we minimize the unbiased estimator of the risk \eqref{URE} over the set $\Lambda$. Now, we consider the second class of shrinkage estimators of the mean parameter. We replace the given location $\mu_j \in \mathbb{R}$ with the mean vector $\bar{Y}_j = \frac{1}{n} \sum_{i=1}^{n} Y_{ij}$, $j=1,\ldots, p$.

The second class of shrinkage estimators is of the form
\begin{align}
\label{estimator_mean}
\hat{\theta}_{ij}^{b,\bar{Y}}=(1-b_i)Y_{ij}+b_i\bar{Y}_j,  
\end{align}
$i=1,\ldots, n, \; j=1, \ldots, p$. 

The squared error loss of the estimators in \eqref{estimator_mean} is 
\begin{align}
\label{loss_2}
\ell_{b}= \frac{1}{np} \sum_{i=1}^{n} \sum_{j=1}^{p} (\hat{\theta}_{ij}^{b,\bar{Y}}-\theta_{ij})^2,
\end{align}
and therefore we define its risk as 
\begin{align}
\label{risk_2}
R_{b}=E(\ell_{b}).
\end{align}

Again, we find an optimal choice of $b$ by minimizing an unbiased estimator of the risk \eqref{risk_2}. In Proposition \ref{proposition_unbiased_2}, we propose an unbiased estimator of the risk.

\begin{proposition}
\label{proposition_unbiased_2}
An unbiased estimator of the risk $R_{b}$ in \eqref{risk_2} is given by
\begin{align}
\label{AURE}
\hat{R}_b = \frac{1}{np} \sum_{i=1}^{n} \sum_{j=1}^{p} \bigg[b_i^2(Y_{ij} - \bar{Y}_j)^2 + \bigg(1-2\bigg(1-\frac{1}{n}\bigg)b_i\bigg)\frac{V(Y_{ij})}{\tau_{ij}+v_2}\bigg].
\end{align}
\end{proposition}

\begin{proof}
Taking the expectation of the estimator in \eqref{AURE} and using \eqref{proposition_proof_0}, we obtain
\begin{align}
\label{proposition2_proof_1}
E(\hat{R}_b)&=\frac{1}{np} \sum_{i=1}^{n} \sum_{j=1}^{p} \bigg[b_i^2 E(Y_{ij} - \bar{Y}_j)^2 +
\bigg(1-2\bigg(1-\frac{1}{n}\bigg)b_i\bigg)\Var(Y_{ij})\bigg] \nonumber \\
&=\frac{1}{np} \sum_{i=1}^{n} \sum_{j=1}^{p} \big[b_i^2 E(Y_{ij} - \bar{Y}_j)^2 +
\Var(Y_{ij})- 2b_i\Var(Y_{ij}) \nonumber \\ 
&\qquad \qquad \qquad \qquad\qquad\qquad\qquad + \frac{2}{n} b_i E(Y_{ij}^2) - \frac{2}{n} b_i[E(Y_{ij})]^2 \big].
\end{align}
Now, note that
\begin{align}
\label{proposition2_risk_1}
R_{b}= \frac{1}{np} \sum_{i=1}^{n} \sum_{j=1}^{p} E(\hat{\theta}_{ij}^{b,\bar{Y}}-\theta_{ij})^2 &= \frac{1}{np} \sum_{i=1}^{n} \sum_{j=1}^{p} E\big( b_i\bar{Y}_j + (1-b_i)Y_{ij} - \theta_{ij}\big)^2 \nonumber \\
&= \frac{1}{np} \sum_{i=1}^{n} \sum_{j=1}^{p} E\big(b_i\bar{Y}_j +Y_{ij}-b_iY_{ij} - \theta_{ij}\big)^2  
\end{align}
By rearranging terms and applying the elementary identity, $(x+y)^2\equiv x^2+y^2+2xy$ for $x,y \in \mathbb{R}$, we derive \eqref{proposition2_risk_1} in the form
\begin{align}
\label{proposition2_risk_2}
R_{b}&= \frac{1}{np}\sum_{i=1}^{n} \sum_{j=1}^{p} \Big[ b_{i}^2E(Y_{ij} - \bar{Y}_{j})^2 +E(\theta_{ij}-Y_{ij})^2 + 2E\big((Y_{ij}-\theta_{ij})(\bar{Y}_{j}b_{i}-Y_{ij}b_i)\big)  \Big] \nonumber \\
&= \frac{1}{np}\sum_{i=1}^{n} \sum_{j=1}^{p} \big[ b_{i}^2E(Y_{ij}- \bar{Y}_{j})^2     +\Var(Y_{ij}) + 2E\big((Y_{ij}-\theta_{ij})(\bar{Y}_{j}b_{i}-Y_{ij}b_i)\big)  \big].
\end{align}
Expanding the last term of equation \eqref{proposition2_risk_2}, we obtain
\begin{align}
\label{proposition2_risk_3}
\sum_{i=1}^{n} \sum_{j=1}^{p} & E\big((Y_{ij}-\theta_{ij})(\bar{Y}_{j}b_{i}-Y_{ij}b_i)\big) \nonumber\\&= \sum_{i=1}^{n} \sum_{j=1}^{p} E\big(Y_{ij}\bar{Y}_{j}b_{i} - Y_{ij}Y_{ij}b_i - \theta_{ij}\bar{Y}_{j}b_{i} + \theta_{ij}Y_{ij}b_i\big) \nonumber\\
&=\sum_{i=1}^{n} \sum_{j=1}^{p} \big[E\big(Y_{ij}\bar{Y}_{j}b_{i} - \theta_{ij}\bar{Y}_{j}b_{i}) - b_i\Var(Y_{ij})\big].
\end{align}
We further expand the first two terms of the equation \eqref{proposition2_risk_3}. We split the first term into two sums for $i=k$ and $i \neq k$, thus
\begin{align*}
\sum_{i=1}^{n} \sum_{j=1}^{p} E\big(Y_{ij}\bar{Y}_{j}b_{i}\big)
&=\frac{1}{n} \sum_{j=1}^{p} \sum_{i=1}^{n} \sum_{k=1}^{n} b_{i} E(Y_{ij}Y_{kj}) \nonumber\\
&=\frac{1}{n} \sum_{j=1}^{p} \sum_{i=1}^{n} b_{i} E(Y_{ij}^2) +\frac{1}{n} \sum_{j=1}^{p} \sum_{i \neq k} b_{i} E(Y_{ij}Y_{kj}),
\end{align*}
and since $Y_{ij}$ are mutually independent for all $i$, we get 
\begin{align}
\label{proposition2_risk_4_1}
\sum_{i=1}^{n} \sum_{j=1}^{p} E\big(Y_{ij}\bar{Y}_{j}b_{i}\big)
&=\frac{1}{n} \sum_{j=1}^{p} \sum_{i=1}^{n} b_{i} E(Y_{ij}^2) +\frac{1}{n} \sum_{j=1}^{p} \sum_{i \neq k} b_{i} E(Y_{ij})E(Y_{kj}).
\end{align}
By using similar arguments, the second term in \eqref{proposition2_risk_3} becomes
\begin{align}
\label{proposition2_risk_5}
\sum_{i=1}^{n} \sum_{j=1}^{p} E\big(\theta_{ij}\bar{Y}_{j}b_{i}\big)
&=\frac{1}{n} \sum_{j=1}^{p} \sum_{i=1}^{n} \sum_{k=1}^{n} b_{i} E(Y_{ij})E(Y_{kj}) \nonumber \\ 
&=\frac{1}{n} \sum_{j=1}^{p} \sum_{i=1}^{n} b_{i} [E(Y_{ij})]^2 +\frac{1}{n} \sum_{j=1}^{p} \sum_{i \neq k} b_{i} E(Y_{ij})E(Y_{kj}).
\end{align}
Substituting \eqref{proposition2_risk_4_1} and \eqref{proposition2_risk_5} in \eqref{proposition2_risk_3}, it follows from \eqref{proposition2_proof_1} and \eqref{proposition2_risk_2} that $E(\hat{R}_b)=R_{b} $ and thus the estimator in \eqref{AURE} is an unbiased estimator of the risk in \eqref{risk_2}. 
\end{proof}

%\medskip

We are interested in finding $\hat{b}^{*}$ that minimizes \eqref{AURE}  over the set 
\begin{align*}
\Psi = \Big\{ b: b_i \in [0,1] 
\Big\},
\end{align*}
and the shrinkage estimator in \eqref{estimator_mean} becomes 
\begin{align}
\label{estimator_mean_AURE}
\hat{\theta}_{i}^{\hat{b}^{*}, \bar{Y}}=(1-\hat{b}_{i}^{*})Y_i+\hat{b}_i^{*} \bar{Y},
\end{align}
$i= 1, \ldots,n$.

Next, we present our main results for the second class of shrinkage estimators given by \eqref{estimator_mean_AURE}.  In Theorem \ref{uniform_AURE}, we show that the estimator $\hat{R}_b$ is uniformly close to the actual loss $\ell_{b}$.

\begin{theorem}
\label{uniform_AURE}
Suppose that the assumptions \textnormal{\ref{A}-\ref{E}} hold. Then, as $n, p \rightarrow\infty$,
\begin{align}
\label{uniform_AURE_result}
E\Big(\sup_{b \in \Psi} |\hat{R}_b-\ell_{b}| \Big)= \textnormal{O}(n^{-1/2}+n^{-\alpha/2}p^{\beta/2}).
\end{align}
\end{theorem}
   
Further, we show that our proposed estimator is asymptotically optimal among a large class of shrinkage estimators.

\begin{theorem}
\label{theorem_optimality_AURE}
Suppose that the assumptions \textnormal{\ref{A}-\ref{E}} hold, and consider any shrinkage estimator of the form 
\begin{align}
%\label{estimator_form_2}
\hat{\theta}_{i}^{\hat{b},\bar{Y}}=(1-\hat{b}_{i})Y_i+\hat{b}_i \bar{Y}, \nonumber
\end{align}
$i=1,\ldots, n,$, where $\hat{b} \in \Psi$. Then, as $n, p \rightarrow \infty$, 
\begin{align}
%\label{optimality_result_3}
\ell_{\hat{b}^ * } \leq \ell_{\hat{b}} + \textnormal{O}_{\mathbb{P}}(n^{-1/2}+n^{-\alpha/2}p^{\beta/2}), \nonumber
\end{align}
and
\begin{align}
%\label{optimality_result_4}
\limsup\limits_{n\rightarrow\infty} (R_{\hat{b}^* } - R_{\hat{b}})\leq 0.\nonumber
\end{align}
\end{theorem}

\medskip

The proofs of Theorems \ref{uniform_AURE} and \ref{theorem_optimality_AURE} are presented at the end of this section. Let us define the quantities
\begin{align*}
&T_4= \frac{2}{n^2p}\sum_{i=1}^{n} \sum_{j=1}^{p} 
(Y_{ij}-\theta_{ij})^2 ,\\
&T_5= \frac{2}{np} \sup_{(b,\mu) \in \Psi} \sum_{j=1}^{p} |\bar{Y}_j| \bigg|\sum_{j=1}^{p} b_i(Y_{ij} - \theta_{ij}) \bigg|.
\end{align*}
In the following result, which is needed to prove Theorem \ref{uniform_AURE} and \ref{theorem_optimality_AURE}, we establish the convergence properties of $T_4$ and $T_5$.

\begin{lemma}
\label{lemma_T4_T5}
Suppose that the assumptions \textnormal{\ref{A}-\ref{E}} hold.  Then, as $n, p \rightarrow\infty$, we have $E(|T_4|)=\textnormal{O}(n^{-1/2}) \text{ and } E(|T_5|)=\textnormal{O}(n^{-\alpha/2}p^{\beta/2})$.
\end{lemma}

\begin{proof}
In bounding $T_4$, we obtain
\begin{align*}
E\bigg( \frac{2}{n^2p} \sum_{i=1}^{n} \sum_{j=1}^{p} (Y_{ij}-\theta_{ij})^2 \bigg) &= \frac{2}{n^2p} \sum_{i=1}^{n} \sum_{j=1}^{p} E(Y_{ij}-\theta_{ij})^2 %\\
%&
=\frac{2}{n^2p} \sum_{i=1}^{n} \sum_{j=1}^{p} \Var(Y_{ij}).
\end{align*}
Under the regularity condition \ref{A}, it follows that, as $n, p \rightarrow \infty$, $E(|T_4|)= O(n^{-1/2})$.

In bounding $T_5$, as in the proof of Lemma \ref{lemma_T1_T2_T3}, we have 
\begin{align*}
 \frac{2}{np} E \bigg( \sup_{(b,\mu) \in \Psi} & \sum_{j=1}^{p}\big|\bar{Y}_j\big|\bigg| \sum_{i=1}^{n} b_i(Y_{ij}-\theta_{ij}) \bigg| \bigg)\\
&\leq \frac{2}{np} \sum_{j=1}^{p} E \bigg( \max\limits_{i,j}\big|Y_{ij}\big|\sup_{(b,\mu) \in \Psi}\bigg| \sum_{i=1}^{n} b_i(Y_{ij}-\theta_{ij}) \bigg| \bigg)\\
&\leq \frac{2}{np} \sum_{j=1}^{p} \bigg[ E \big( \max\limits_{i,j}Y_{ij}^2 \big) E \bigg( \sup_{(b,\mu) \in \Psi}\bigg| \sum_{i=1}^{n} b_i(Y_{ij}-\theta_{ij}) \bigg|^2 \bigg) \bigg]^{1/2},
\end{align*}
and thus $E(|T_5|)= O(n^{-\alpha/2}p^{\beta/2})$, which completes the proof.
\end{proof}

\medskip 

We now provide the proofs of Theorems \ref{uniform_AURE} and \ref{theorem_optimality_AURE}.

\medskip 

\begin{proof}[Proof of Theorem \ref{uniform_AURE}]  By using \eqref{estimator_mean}, \eqref{loss_2}, and \eqref{AURE}, we obtain
\begin{align}
\label{uniform_AURE_proof_1}
\hat{R}_b-\ell_{b} &=\frac{1}{np} \sum_{i=1}^{n} \sum_{j=1}^{p} 
\bigg(1-2\bigg(1-\frac{1}{n}\bigg)b_i\bigg)\bigg[ \frac{V(Y_{ij})}{\tau_{ij}+v_2}-(Y_{ij}-\theta_{ij})^2 \bigg] \nonumber\\
&\quad + \frac{2}{np} \sum_{i=1}^{n} \sum_{j=1}^{p} 
b_i\bigg(\theta_{ij}(Y_{ij}-\theta_{ij})+ \frac{1}{n}(Y_{ij}-\theta_{ij})^2-(Y_{ij} - \theta_{ij})\bar{Y}_j \bigg).
\end{align}

\noindent Taking the supremum over $\Psi$ of the absolute value of the eq. \eqref{uniform_AURE_proof_1} , we have
\begin{align*}
\sup_{(b,\mu) \in \Psi}| \hat{R}_b-\ell_{b}|
&\leq  T_1+\bigg(1-\frac{1}{n}\bigg)T_2+T_{31}+ T_4+  T_5.
\end{align*}
By taking expectations on both sides of the above expression and applying Lemmas \ref{lemma_T1_T2_T3} and \ref{lemma_T4_T5}, it follows that \eqref{uniform_AURE_result} holds. 
\end{proof}

\smallskip 

\begin{proof}[Proof of Theorem \ref{theorem_optimality_AURE}]  With Theorem \ref{uniform_AURE} established, the proof of Theorem \ref{theorem_optimality_AURE} is almost identical to that of Theorem \ref{theorem_optimality}. 
\end{proof}

%----------------------------------------------------------------------------------------%

\section{Shrinkage estimation for the diagonal multivariate natural exponential families}
\label{sec_mnef}

In this section, we focus on the diagonal multivariate natural exponential families and simplify conditions \ref{A}-\ref{E} for those families. \cite{bar1994diagonal} showed that there are six irreducible, diagonal natural exponential families in $\mathbb{R}^p$. These families are the familiar multivariate normal, Poisson, gamma, multinomial, and negative multinomial distributions, and an additional exceptional family called the hybrid distributions. In this section, we focus on the first five of these families.

\iffalse
We remind the reader of the setting in which our results are derived. Let $\Theta \subseteq \mathbb{R}$ be the space of all possible values of our parameters.  For $i=1,\ldots,n$, let $Y_i=(Y_{i1},\ldots,Y_{ip}) \in \mathbb{R}^p$ be a random vector, with distribution function $F_i$, that belongs to a diagonal multivariate natural exponential family and having finite mean and covariance matrix.  We suppose that $Y_1, \ldots,Y_n$ are mutually independent random vectors, we denote the mean of each $Y_i$ by $E(Y_i) := \theta_i =(\theta_{i1},\ldots,\theta_{ip})$, where the unknown parameters $\theta_{ij} \in \Theta$ for all $i=1,\ldots,n$ and $j = 1,\ldots,p$. For known constants $\nu_0, \nu_1, \nu_2 \in \mathbb{R}$, define the function $V(t) = \nu_0+\nu_1 t+\nu_2 t^2$, $t \in \Theta$. We assume that $\nu_0, \nu_1, \nu_2$ are such that $V(t) > 0$ for all $t \in \Theta$. We further assume that, for each $i=1,\ldots,n$, the diagonal entries of the covariance matrix of the distribution $F_i$ are of the form 
\begin{align*}
\text{diag }\big(\Cov(Y_i)\big) &:= \big(\Var(Y_{i1}),\ldots,\Var(Y_{ip})\big)
= \bigg(\frac{V(\theta_{i1})}{\tau_{i1}},\ldots,\frac{V(\theta_{ip})}{\tau_{ip}} \bigg) \;,
\end{align*}
where the known constants $\tau_{ij}  \in \mathbb{N}$.
\fi

In simplifying conditions \ref{A}-\ref{E}, we modify and apply a result of \cite[Lemma A.1, p.~593]{xie2016optimal}.  In stating the following result, we recall that if a random variable $X$ has mean $\mu$ and variance $\sigma^2$ then the {\it skew} of $X$ is $E(X-\mu)^3/\sigma^3$.  We introduce the following assumptions:
\begin{enumerate}[label=(\Alph*)]\addtocounter{enumi}{5}
\item \label{F}$ \limsup\limits_{n, p \rightarrow \infty} \sum_{i=1}^{n} \sum_{j=1}^{p}|\theta_{ij}|^{2+\tilde{\epsilon}}/np<\infty$ for some $\tilde{\epsilon}>0$,
\item \label{G}$ \limsup\limits_{n, p \rightarrow \infty}\sum_{i=1}^{n}\sum_{j=1}^{p} (\Var(Y_{ij}))^2/np < \infty$,
\item \label{H}$\sup\limits_{i} \textnormal{skew}(Y_{ij}) = \sup\limits_{i} \Big( \big(\nu_1+2\nu_2 \theta_{ij})/\big( \tau^{1/2}_{ij} \, (\nu_0+\nu_1\theta_{ij}+\nu_2\theta_{ij}^2)^{1/2}\big) \Big)< \infty$ for all  $j$,
\end{enumerate}

\begin{lemma}
\label{sufficient_conditions_E}
Let $Y_1, \ldots, Y_n$ be mutually independent random vectors with $Y_{ij}$ coming from one of the five natural exponential families with quadratic variance functions. Then conditions \textnormal{\ref{B}} and \textnormal{\ref{F}-\ref{H}}, imply condition \textnormal{\ref{E}} if (i) $0<\tilde{\epsilon}<2$, $pn^{-\tilde{\epsilon}/2}\rightarrow 0$ as $n,p\rightarrow \infty$, or (ii) $\tilde{\epsilon}\ge 2$, $pn^{-1/3}\log^{4/3}n\rightarrow 0$ as $n,p\rightarrow\infty$.
\end{lemma}

\begin{proof} The proof is similar to the proof of Lemma A.1 in \cite{xie2016optimal}, where $Y_{ij}=\sigma_{ij}Z_{ij}+\theta_{ij}$ with $\sigma_{ij}^2=\Var(Y_{ij})$ and $Z_{ij}$ are independent for all $i$ with mean zero and variance one. Since $Y_{ij}^2=\sigma_{ij}^2Z_{ij}^2+\theta_{ij}^2+2\sigma_{ij}\theta_{ij}Z_{ij}$, we get that
\begin{align*}
\max \limits_{i,j}Y_{ij}^2 = \max \limits_{i,j}\sigma_{ij}^2 \cdot \max \limits_{i,j}Z_{ij}^2 + \max \limits_{i,j}\theta_{ij}^2+2\max \limits_{i,j}\sigma_{ij}|\theta_{ij}| \cdot \max \limits_{i,j}|Z_{ij}|.
\end{align*}
Using the hypotheses of Lemma \ref{sufficient_conditions_E}, as well as Lemma A.1 in \cite{xie2016optimal}, we have
\begin{align*}
E(\max \limits_{i,j}Y_{ij}^2)&=\text{O}\big((np)^{1/2}p\log^2n+(np)^{2/(2+\tilde{\epsilon})}+(np)^{1/2}p\log n\big) \\
&= \begin{cases}
      \text{O}\big((np)^{2/(2+\tilde{\epsilon})}\big), & 0<\tilde{\epsilon} < 2 \\
      \text{O}\big((np)^{1/2}p\log^2n\big), & \tilde{\epsilon} \geq  2
\end{cases} .
\end{align*}
For $0<\tilde{\epsilon}<2$, we deduce that 
$$
\frac{1}{n} E(\max\limits_{i,j} Y_{ij}^2)=\text{O}(n^{-\tilde{\epsilon} / (2+\tilde{\epsilon})} p^{2/(2+\tilde{\epsilon})}).
$$
Therefore, for \ref{E} to hold, we require that $pn^{-\tilde{\epsilon}/2}\rightarrow 0$ as $n,p\rightarrow\infty$.
For $\tilde{\epsilon} \geq 2$, we deduce that 
$$
\frac{1}{n} E(\max\limits_{i,j} Y_{ij}^2)= \text{O}\big(n^{-1/2}p^{3/2} \log^2n\big),
$$
so that $pn^{-1/3}\log^{4/3}n\rightarrow 0$ implies \ref{E}.
\end{proof}

\medskip

For the five diagonal multivariate natural exponential families that we consider in this article, below we list the respective conditions, under which regularity conditions \ref{A}-\ref{E} are satisfied.

\begin{proposition} 
\label{proposition_simplifiedA-E}
For the normal distribution with $E(Y_{ij})=\theta_{ij}$ and $\textnormal{Var} (Y_{ij})= 1$, where $\theta_{ij} \in (-\infty,\infty)$, conditions \textnormal{\ref{A}-\ref{E}} reduce to: 
$\sum_{i=1}^{n}\sum_{j=1}^{p} \theta_{ij}^4 =\textnormal{O}(np).$

%\smallskip
For the Poisson distribution with $E(Y_{ij})=\textnormal{Var}(Y_{ij})=\theta_{ij}$, where $\theta_{ij} >0$, conditions \textnormal{\ref{A}-\ref{E}} reduce to: 
$\sum_{i=1}^{n} \sum_{j=1}^{p} \theta_{ij}^3 =\textnormal{O}(np) \text{ and }\inf_{i} \theta_{ij}>0,$ 
$j=1, \ldots, p$.

%\smallskip
For the gamma distribution with $E(Y_{ij})=\theta_{ij}$ and $\textnormal{Var}(Y_{ij})=\theta_{ij}^2/\lambda$, where $\theta_{ij}, \lambda > 0$, conditions \textnormal{\ref{A}-\ref{E}} reduce to: 
$\sum_{i=1}^{n} \sum_{j=1}^{p} \theta_{ij}^4 = \textnormal{O}(np).$ 

%\smallskip
For the multinomial distribution with $E(Y_{ij})=\theta_{ij}$ and $\textnormal{Var}(Y_{ij}) = (\theta_{ij}-\theta_{ij}^2)/N_{i}$, where $\theta_{ij} \in (0,1)$ and $N_i \geq 2$, conditions \textnormal{\ref{A}-\ref{E}} reduce to: 
$\sum_{i=1}^{n} \sum_{j=1}^{p}\theta_{ij}^3 = \textnormal{O}(np)$ and $\inf_{i} \theta_{ij}>0,$ 
$j=1, \ldots, p$.

%\smallskip
For the negative multinomial distribution with $E(Y_{ij})=\theta_{ij}$, $\textnormal{Var}(Y_{ij})=(\theta_{ij}+\theta_{ij}^2)/N_i$, $\theta_{ij}>0$, and $N_i \in \mathbb{N}$, conditions \ref{A}-\ref{E} reduce to: 
$\sum_{i=1}^{n}\sum_{j=1}^{p} \theta_{ij}^4 = \textnormal{O}(np)$ and $\inf\limits_{i} \theta_{ij}>0$, $j=1, \ldots, p$.
\end{proposition}

\begin{proof}  For the multivariate normal distribution, we have $E(W_{ij})=\theta_{ij}'$ 
and $\Var(W_{ij})= \sigma_{ij}^2$.  
Instead of the variables $W_{ij}$, we consider the transformation $Y_{ij}=W_{ij}/(\sigma_{ij}^2)^{1/2}$. Then, we have $E(Y_{ij})=\theta_{ij}'/(\sigma_{ij}^2)^{1/2}=\theta_{ij} \text{ for $\theta_{ij} \in \mathbb{R}$}$ and $\Var(Y_{ij})=1$, i.e., $\nu_0=1$, $\nu_1=\nu_2=0$ and $\tau_{ij}=1$.  Therefore, Condition \ref{A} is satisfied directly. Using the Jensen's inequality, we have
\begin{align*}
 \frac{1}{np} \sum_{i=1}^{n}\sum_{j=1}^{p} \theta_{ij}^2 \leq \bigg( \frac{1}{np}\sum_{i=1}^{n}\sum_{j=1}^{p}  \theta_{ij}^4\bigg)^{1/2},
\end{align*}
and
\begin{align*}
\frac{1}{np} \sum_{i=1}^{n}\sum_{j=1}^{p} \Var(Y_{ij}^2) &= \frac{1}{np} \sum_{i=1}^{n}\sum_{j=1}^{p}  \big[E(Y_{ij}^4) - \big(E(Y_{ij}^2)\big)^2\big] \\
&= \frac{1}{np} \sum_{i=1}^{n}\sum_{j=1}^{p} \big(2+4\theta_{ij}^2\big).
\end{align*}
Therefore, it is straightforward to verify that $\sum_{i=1}^{n}\sum_{j=1}^{p} \theta_{ij}^4 =O(np)$ imply conditions \ref{B} and \ref{C}. Condition \ref{D} is also immediately satisfied. Note each $Y_{ij}$ follows a univariate normal distribution. Thus, condition \ref{F} is trivially satisfied for $\tilde{\epsilon}=2$. Also, conditions \ref{G} and \ref{H} are satisfied.

For the multivariate Poisson distribution, since $E(Y_{ij})=\Var(Y_{ij})= \theta_{ij}$ for $\theta_{ij} >0$, we have $\nu_1=1$, $\nu_0 = \nu_2=0$, and $\tau_{ij}=1$. By Jensen's inequality, we obtain
\begin{align*}
\frac{1}{np} \sum_{i=1}^{n} \sum_{j=1}^{p} \Var(Y_{ij}) &= \frac{1}{np} \sum_{i=1}^{n}\sum_{j=1}^{p}  \theta_{ij} \leq \bigg(\frac{1}{np} \sum_{i=1}^{n} \sum_{j=1}^{p} \theta_{ij}^3 \bigg)^{1/3}.
\end{align*}
Therefore, it is straightforward to see that $\sum_{i=1}^{n} \sum_{j=1}^{p} \theta_{ij}^3 =O(np)$ implies conditions \ref{A}. Condition \ref{B} is satisfied directly. Since
\begin{align*}
\frac{1}{np} \sum_{i=1}^{n} \sum_{j=1}^{p} \Var(Y_{ij}^2) &= \frac{1}{np} \sum_{i=1}^{n}\sum_{j=1}^{p}  \big[E(Y_{ij}^4) - \big(E(Y_{ij}^2)\big)^2\big] \\
&= \frac{1}{np} \sum_{i=1}^{n} \sum_{j=1}^{p} \big(\theta_{ij}+6\theta_{ij}^2+4\theta_{ij}^3\big),
\end{align*}
we obtain \ref{C} by arguments similar to those in the proof of \ref{A}. Condition \ref{D} is also immediately satisfied. Note that each $Y_{ij}$ follows a univariate Poisson distribution. Thus, condition \ref{F} is satisfied for $\tilde{\epsilon}=1$. Also, condition \ref{G} is satisfied. Since $\sup_{i} \text{skew}(Y_{ij})= \sup_{i} (1/\theta_{ij}^{1/2})$, and by the assumption that $\inf_{i} \theta_{ij}>0$, we obtain \ref{H}.

For the multivariate gamma distribution, since $E(Y_{ij})=\theta_{ij}$ and $\Var(Y_{ij})=\theta_{ij}^2/\lambda$ for $\theta_{ij}, \lambda > 0$, we obtain $\nu_0=\nu_1=0$, $\nu_2=1/\lambda$, and $\tau_{ij}=1$. Using Jensen's inequality, we obtain
\begin{align*}
\frac{1}{np} \sum_{i=1}^{n} \sum_{j=1}^{p} \Var(Y_{ij}) &= \frac{1}{np} \sum_{i=1}^{n}\sum_{j=1}^{p}  \frac{\theta_{ij}^2}{\lambda} \leq \bigg(\frac{1}{np} \sum_{i=1}^{n} \sum_{j=1}^{p} \frac{\theta_{ij}^4}{\lambda^2}
 \bigg)^{1/2}.
\end{align*}
Therefore, it is straightforward to verify that $\sum_{i=1}^{n}\sum_{j=1}^{p}  \theta_{ij}^4 = O(np)$ implies conditions \ref{A}. Condition \ref{B} is satisfied directly.
Since
\begin{align*}
\frac{1}{np} \sum_{i=1}^{n}\sum_{j=1}^{p} \Var(Y_{ij}^2)&=\frac{1}{np} \sum_{i=1}^{n}\sum_{j=1}^{p}  \big[E(Y_{ij}^4) - \big(E(Y_{ij}^2)\big)^2\big] \\
&\leq \frac{1}{np} \sum_{i=1}^{n}\sum_{j=1}^{p}  \bigg(\frac{6}{\lambda^3}+\frac{3}{\lambda^2}+\frac{4}{\lambda}+1\bigg)\theta_{ij}^4,
\end{align*}
we can prove that condition \ref{C} holds. Condition \ref{D} is also immediately satisfied. Note that each $Y_{ij}$ follows a univariate gamma distribution. Thus, condition \ref{F} is satisfied for $\tilde{\epsilon}=2$. Also, condition \ref{G} is satisfied. Since  $\sup_{i}\text{skew}(Y_{ij})= \sup_{i}(2/\lambda^{1/2})$, and $\lambda>0$, we obtain \ref{H}.

In the case of the multivariate multinomial distribution, we have $E(W_{ij})=\theta_{ij}'$ and $\Var(W_{ij}) =\theta_{ij}'-(\theta_{ij}'^2/N_{i})$. Instead of the variables $W_{ij}$, we consider the transformation $Y_{ij}=W_{ij}/N_i$. Then, we have $E(Y_{ij})=\theta_{ij}'/N_i=\theta_{ij}$ and $\Var(Y_{ij})=(\theta_{ij}-\theta_{ij}^2)/N_{i}$, i.e., $\nu_0=0$, $\nu_1=1$, $\nu_2=-1$ and $\tau_{ij}=N_i$. Using Jensen's inequality, we can show condition \ref{A}. Condition \ref{B} is also satisfied.
Since
\begin{align*}
\frac{1}{np} \sum_{i=1}^{n}\sum_{j=1}^{p}  \Var(Y_{ij}^2)&=\frac{1}{np} \sum_{i=1}^{n} \sum_{j=1}^{p} \big[E(Y_{ij}^4) - \big(E(Y_{ij}^2)\big)^2\big] \\
& \leq \frac{1}{np} \sum_{i=1}^{n} \sum_{j=1}^{p} \bigg[\bigg(\frac{5}{N_i^3}\bigg)\theta_{ij}+\bigg(\frac{5}{N_i}-\frac{4}{N_i^2}+\frac{16}{N_i^3}\bigg)\theta_{ij}^3\bigg],
\end{align*}
it is straightforward to verify that $\sum_{i=1}^{n} \sum_{j=1}^{p} \theta_{ij}^3 =O(np)$ and $N_i\geq 2$ imply conditions \ref{C}. Condition \ref{D} is also satisfied for $N_i\geq 2$ . Note that each $Y_{ij}$ follows a Binomial distribution. Thus, condition \ref{F} is satisfied for $\tilde{\epsilon}=1$. Also, condition \ref{G} is satisfied using similar arguments. Since 
\begin{align*}
\sup\limits_{i}\text{skew}(Y_{ij})= \sup\limits_{i}\frac{1}{N_{i}^{1/2}}\frac{1-2\theta_{ij}}{(\theta_{ij}-\theta_{ij}^2)^{1/2} }\leq \sup\limits_{i}\frac{1}{(\theta_{ij}-\theta_{ij}^2)^{1/2}},
\end{align*}
and by assuming that $\inf_{i} \theta_{ij}>0$, we prove \ref{H}.

For the multivariate negative multinomial distribution, we have $E(W_{ij})=\theta_{ij}'$ and
$\Var(W_{ij})=\theta_{ij}'+(\theta_{ij}'^2/N_{i})$. Instead of the variables $W_{ij}$, we consider the transformation $Y_{ij}=W_{ij}/N_i$. Then, we have $E(Y_{ij})=\theta_{ij}'/N_i=\theta_{ij}$ and  
$\Var(Y_{ij})=(\theta_{ij}+\theta_{ij}^2)/N_{i} \text{ for $\theta_{ij} >0$ and $N_i \in \mathbb{N}$}$, i.e., $\nu_0=0$, $\nu_1=1, \nu_2=1$ and $\tau_{ij}=N_i$.
We can prove \ref{A} and \ref{B} by using similar arguments as those in the multivariate multinomial distribution.
Since
\begin{align*}
\frac{1}{np} \sum_{i=1}^{n} \sum_{j=1}^{p} \Var(Y_{ij}^2)&=\frac{1}{np} \sum_{i=1}^{n}\sum_{j=1}^{p} \big[E(Y_{ij}^4) - \big(E(Y_{ij}^2)\big)^2\big] \\
& \leq \frac{1}{np} \sum_{i=1}^{n}\sum_{j=1}^{p}  \big[a_1\theta_{ij}+a_2\theta_{ij}^2+a_3\theta_{ij}^3+a_4\theta_{ij}^4\big],
\end{align*}
where $a_k>0, k=1,2,3,4$, it is straightforward to verify that $\sum_{i=1}^{n} \sum_{j=1}^{p} \theta_{ij}^4 =O(np)$ implies conditions \ref{C}. Condition \ref{D} is satisfied directly. Note that each $Y_{ij}$ follows a negative binomial distribution. Thus, condition \ref{F} is automatically satisfied for $\tilde{\epsilon}=2$. Also, the condition \ref{G} is immediately satisfied. Since
\begin{align*}
\sup\limits_{i}\text{skew}(Y_{ij})= \sup\limits_{i}\frac{1}{N_{i}^{1/2}}\frac{1+2\theta_{ij}}{(\theta_{ij}+\theta_{ij}^2)^{1/2}} ,
\end{align*}
and by assuming that $\inf_{i} \theta_{ij}>0$, we get condition \ref{H}. 
\end{proof}

\medskip 

Since Proposition \ref{proposition_simplifiedA-E} ensures conditions \ref{A}-\ref{E}, it follows that Theorems \ref{uniform_URE}, \ref{theorem_optimality}, \ref{uniform_AURE}, and \ref{theorem_optimality_AURE} hold for the diagonal multivariate natural exponential families.

\begin{remark}
{\rm
We note that for each one of the five diagonal multivariate natural exponential families in Proposition \ref{proposition_simplifiedA-E}, the rate of convergence is controlled by the sum of the third or fourth power of the mean parameters. For the normal, gamma and negative multinomial distributions, we have that 
 $\sum_{i=1}^{n}\sum_{j=1}^{p} \theta_{ij}^4 =\textnormal{O}(np)$, i.e., $\tilde{\epsilon}=2$. Therefore, the rate of convergence becomes $\textnormal{O}(n^{-1/2}+n^{-1/4}p^{3/4}\log{n})$, i.e., $p$ should grow slower than $n^{1/3}(\log n)^{-4/3}$. For the Poisson and multinomial distributions, we have that 
 $\sum_{i=1}^{n}\sum_{j=1}^{p} \theta_{ij}^3 =\textnormal{O}(np)$, i.e., $\tilde{\epsilon}=1$. Therefore, the rate of convergence becomes $\textnormal{O}(n^{-1/2}+n^{-1/6}p^{1/3})$, i.e., $p$ should grow slower than $n^{1/2}$.
 
}
\end{remark}

%-------------------------------------------------------------------------------------------------------

\section{Simulation studies}
\label{sec_ss}

This section provides simulations to test the performance of our proposed mean shrinkage estimators. We conduct the simulations for the multivariate gamma and Poisson cases. We remark that for the other members of the diagonal multivariate natural exponential families, the simulation techniques will be similar. 

We compare four estimators. The first estimator is given by \eqref{estimator_mu_URE}, which shrinks $Y_{ij}$ toward a given location $\hat{\mu}_j$. The second estimator is given by \eqref{estimator_mean_AURE}, which shrinks $Y_{ij}$ toward the grand mean $\bar{Y}_j$. The third estimator is the naive estimator given by
\begin{equation}
%\label{naive_estimator}
\hat{\theta}_{ij}=Y_{ij},   \nonumber
\end{equation}
which is the maximum likelihood estimator for the parameters. Let $\tilde{b}_i^*$ and $\tilde{\mu}_j^*$ be the minimizers of the risk $R_{b,\mu}$ given in \eqref{risk_1} over the set $\Lambda$ given in \eqref{min_set}; then the fourth and last estimator is the oracle estimator given by
\begin{equation}
%\label{oracle_estimator}
\hat{\theta}_{ij}^{\tilde{b}^*,\tilde{\mu}^*}=(1-\tilde{b}_i^*)Y_{ij}+\tilde{b}_i^*\tilde{\mu}_j^*,   \nonumber
\end{equation}
$i=1,\ldots,n, \; j=1,\ldots, p$. Note that the oracle estimator depends on the unknown parameter $\theta_{ij}$, however, it provides us with a lower bound on the risk.

\subsection{The multivariate gamma distribution}

Let $\mathcal{J}$ be the collection of all subsets of $\{1, \ldots, p \}$. For $s = (s_1,\ldots,s_p) \in \mathbb{R}^p$ and $T \in \mathcal{J}$, define 
$$
s^T := \prod_{j \in T} s_j
$$
and 
$$
s_T := \sum_{j \in T} s_j.
$$
In particular, $s^\emptyset = 1$ and $s_\emptyset = 0$.  
Let $c: \mathcal{J} \to \mathbb{R}$ be a mapping such that $c(\emptyset) = 1$.  We shall write $c_T$ as shorthand for $c(T)$, and we refer to $c_T$ as a \textit{coefficient function}.  

Following \cite{bar1994diagonal}, we say that a random vector $Y=(Y_1, \ldots, Y_p) \in \mathbb{R}^p$ has a {\it multivariate gamma distribution} if there exists $\lambda > 0$ and a coefficient function $c : \mathcal{J} \to \mathbb{R}$ such that the Laplace transform of $Y$ exists in a neighborhood of the origin and is of the form 
\begin{equation}
\label{lt_gamma}
L_{\mu}(\eta) = \E\bigg[\exp\bigg(-\sum_{j=1}^p \eta_j Y_j\bigg)\bigg] = \bigg( \sum_{T \in \mathcal{J}} c_T \eta^T \bigg)^{-\lambda},
\end{equation}
for all sufficiently small $\eta_j > 0$, $j=1,\ldots,p$.

In the following result, we derive the probability density function of a multivariate gamma distribution with Laplace transform given by \eqref{lt_gamma}. There are several articles which have studied the bivariate case, see \cite{chatelain2006parameter}, \cite{chatelain2007bivariate}, and \cite{letac2008laplace}, however it appears that our results are the first to resolve the general $p$-dimensional case.

We make use of the generalized hypergeometric series.  For $b \in \mathbb{C}$ and any nonnegative integer $i$, the \textit{rising factorial} is defined as $(b)_0 = 1$ and $(b)_i = b(b+1)\cdots (b+i-1)$ for $i \ge 1$.  Let $r, q$ be non-negative integers, and let $z,b_1,\ldots,b_r,d_1,\ldots,d_q \in \mathbb{C}$ such that $(d_j)_i \neq 0$ for all $j=1,\ldots,q$ and all $i=0,1,2,,\ldots$.  Then the \textit{generalized hypergeometric series} is defined as 
$$
{}_rF_{q}(b_1, \ldots, b_r ; d_1, \ldots, d_q ; z) := \sum_{i=0}^{\infty} \frac{(b_1)_i \ldots (b_r)_i \, z^i}{(d_1)_i \ldots (d_q)_i \, i!}.
$$
The convergence properties of this series are well-known; see \cite{andrews1999special}.  In particular, it is known that the series ${}_0F_q(d_1, \ldots, d_q ; z)$ converges for all $z \in \mathbb{C}$.  

Consider a random vector $Y=(Y_1, \ldots, Y_p) \in \bR^p$ whose distribution depends on a parameter vector, $\beta = (\beta_1,\ldots,\beta_p) \in \bR_+^p$, the positive orthant in $\bR^p$.   

\begin{proposition}
\label{proposition_pdf}
For $\beta_\star := \prod_{j=1}^p \beta_j > 1$, suppose that the random vector $Y$ has probability density function 
\begin{equation}
\label{mgd_pdf}
f(y_1,\ldots,y_p) = C_0(\lambda;\beta) \bigg(\prod_{j=1}^p \frac{y_j^{\lambda-1} e^{-\beta_j x_j}}{\Gamma(\lambda)}\bigg) \, {}_0F_{p-1}(\lambda,\ldots,\lambda;y_1\cdots y_p),
\end{equation}
$y_1,\ldots,y_p > 0$, where 
\begin{equation}
\label{C0_constant}
C_0(\lambda;\beta) = (\beta_\star - 1)^\lambda
\end{equation}
is the normalizing constant.  Then the Laplace transform of $Y$ exists and is of the form \eqref{lt_gamma}, and the corresponding coefficient function in is 
\begin{equation}
\label{coeff_function}
c_T = \begin{cases}
\dfrac{\beta^{\bar{T}}}{\beta_\star - 1}, & T \neq \emptyset \\
1, & T = \emptyset
\end{cases}.
\end{equation}
\end{proposition}

We provide the proof of Proposition \ref{proposition_pdf} in Appendix \ref{appendix_B}.

%\bigskip

\begin{figure}[t!]
\centering
\includegraphics[width=\textwidth, height=10.2cm]{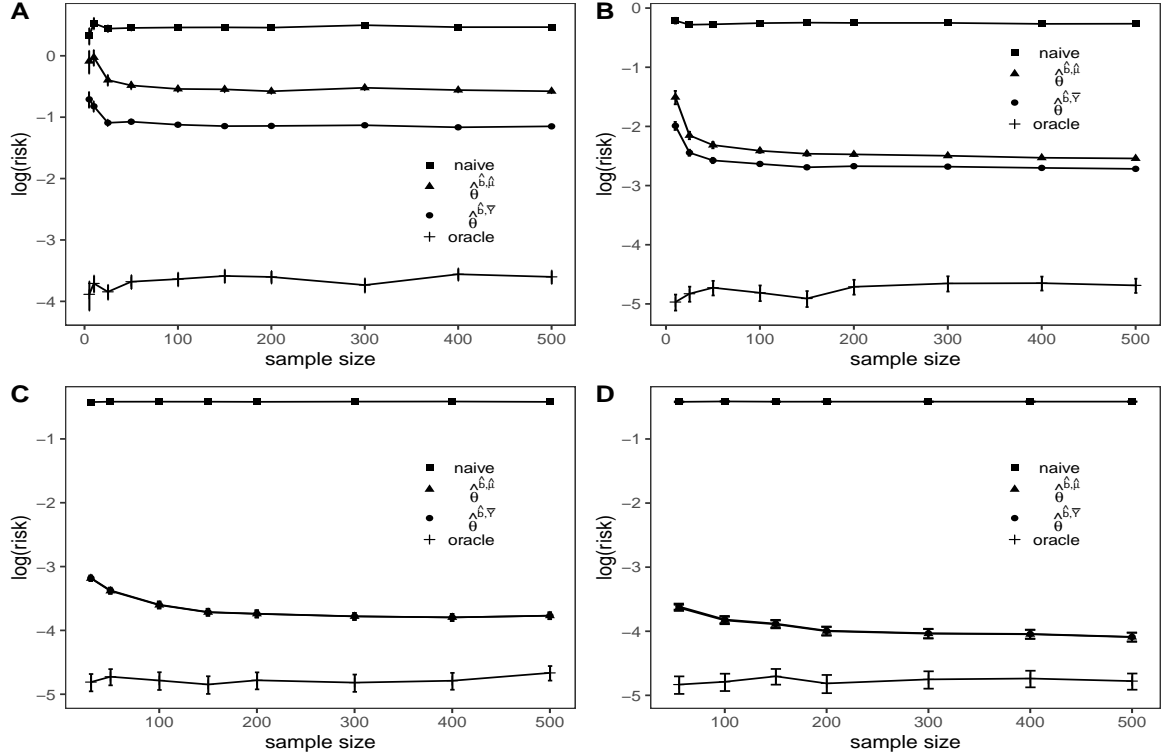}
\caption{Comparisons of the average risk of different estimators for gamma distribution using variable-at-a-time Metropolis algorithm. \textbf{A}: $p=2$. \textbf{B}: $p=5$. \textbf{C}: $p=25$. \textbf{D}: $p=50$.}
\label{fig:gamma_mhv}
\end{figure}

%\bigskip

In our simulation studies the dimension, $p$, is chosen to be  2, 5, 25, and 50. The sample size, $n$, is chosen to vary from values greater than $p$ to 500. We draw $\lambda$ from a univariate inverse gamma distribution with shape parameter, $a_0=11$, and rate parameter, $a_1=20$. We form a vector $\beta_{i}$ for $i=1, \ldots, n$, where each of its element $\beta_{ij}$, $j=1,\ldots, p$ is drawn from a univariate gamma distribution with shape parameter, $b_0=20$, and rate parameter, $b_1=10$. Given $\beta_{i}$ and $\lambda$, we use a variable-at-a-time Metropolis algorithm to generate a sample vector $Y_{i}$ from the multivariate gamma distribution with parameters $\beta_{i}$ and $\lambda$ with p.d.f. given by \eqref{mgd_pdf}. The proposal distribution is a univariate Gaussian distribution with mean set at current state of the chain and variance equals to 9.  

After we obtain the sample $Y_1, \ldots, Y_n$, we calculate the mean estimators, as well as their corresponding risk. In order to obtain accurate estimates for the risks, we repeat the above process 200 times and calculate the average risk for each estimator. 

In Figure \ref{fig:gamma_mhv}, Panels A-D show the logarithm of the average risk of the estimators with 95\% confidence interval for dimension 2, 5, 25, and 50.  We plot the logarithm of the risks to visualize more accurately the size of the risks. The logarithm of the risk of the naive estimator $\hat{\theta}_{ij}$, which is relatively constant for all $n$ as expected, is larger than the other three estimators. The performance of the two shrinkage estimators $\hat{\theta}_{ij}^{\hat{b}^*,\bar{Y}}$ and $\hat{\theta}_{ij}^{\hat{b}^*,\hat{\mu}^*}$ is very good as they are close to the oracle estimators. The shrinkage estimator $\hat{\theta}_{ij}^{\hat{b}^*,\bar{Y}}$ performs better than $\hat{\theta}_{ij}^{\hat{b}^*,\hat{\mu}^*}$ for lower dimensions. However, as we increase the dimension, we observe that the performances of these two estimators are approaching each other.  In panels C and D, the simulated logarithmic risks value for the estimators $\hat{\theta}_{ij}^{\hat{b}^*,\bar{Y}}$ and $\hat{\theta}_{ij}^{\hat{b}^*,\hat{\mu}^*}$ are so close that their plots are nearly coincident.

\subsection{The multivariate Poisson distribution}

According to \cite{bar1994diagonal}, a random vector $Y=(Y_1,\ldots, Y_p) \in \bR^p$ is said to have a {\it multivariate Poisson distribution} if there exists a coefficient function $c : \mathcal{J} \to \bR$ such that the Laplace transform of $Y$ exists in a neighborhood of the origin and is of the form 
\begin{equation}
\label{lt_poisson}
L_{\mu}(\eta)=\E\bigg[\exp\bigg(-\sum_{j=1}^p \eta_j x_j\bigg)\bigg] =\exp\bigg(\sum_{T \subseteq \mathcal{J}, T \neq \emptyset} c_T (e^{(-\eta)_T} - 1)\bigg),
\end{equation}
for all sufficiently small $\eta_j > 0$, $j=1,\ldots,p$.

In the next result, we derive the probability density function of a Poisson distribution for $p \in \mathbb{N}$. We consider a random vector $Y \in \mathbb{R}^p$ whose distribution depends on a parameter vector $c = (c_{1},\ldots,c_{1 \cdots p}) \in \mathbb{R}^{2^p-1}_{+}$.

\begin{proposition}
\label{proposition_pmf_poisson}
For $c = (c_{1},\ldots,c_{1 \cdots p}) \in \mathbb{R}^{2^p-1}_{+}$, suppose that the random vector $Y$ has probability density function 
\begin{align}
\label{poisson_pdf_pd}
P(&Y_1=l_1, \ldots ,Y_p=l_p) \nonumber \\
&= \sum_{\substack{j_{1\bullet}=l_1 \\ \cdots  \\ j_{p\bullet}=l_p}} \bigg(\prod_{l=1}^p \frac{e^{-c_l} c_l^{j_l}}{j_l!}\bigg) \cdot \bigg(\prod_{1 \le l < m \le p} \frac{e^{-c_{lm}} c_{lm}^{j_{lm}}}{ j_{lm}! }\bigg) \cdots \bigg(\frac{ e^{-c_{1\cdots p}} c_{1 \cdots p}^{j_{1 \cdots p}} }{ j_{1 \cdots p}! }\bigg),
\end{align}
$l_1,\ldots,l_p \in \mathbb{N}_{0} $, where 
$$
j_{k\bullet} = \sum_{T \in \mathcal{J} : \, k \in T} j_T,
$$
for each $k=1,\ldots,p$. Then, the Laplace transform of $Y$ exists and is of the form \eqref{lt_poisson}.
\end{proposition}

We provide the proof of Proposition \ref{proposition_pmf_poisson} in Appendix \ref{appendix_C}.

Now, define for each non-empty $T \in \mathcal{J}$ mutually independent random variables $X_T$, where $X_T$ has a Poisson distribution with parameter $c_T$.  Also define 
$$
X_{j\bullet} = \sum_{T \in \mathcal{J} : \, j \in T} X_T,
$$
$j=1,2,3$.  Then it follows from \eqref{poisson_pdf_pd} that $(Y_1, \ldots ,Y_p) \steq (X_{1\bullet}, \ldots ,X_{p\bullet})$.  To see this, observe that 
\begin{align*}
P(X_{1\bullet}=l_1, \ldots ,X_{p\bullet}=l_p) = \sum_{\substack{j_{1\bullet}=l_1 \\ \cdots \\ j_{p\bullet}=l_p}}  P(X_1=j_1, \ldots,X_{1\cdots p}=j_{1 \cdots p}).
\end{align*}
Since the variables $X_T$ are mutually independent then we obtain 
\begin{align*}
P(X_{1\bullet}=l_1,&\ldots,X_{p\bullet}=l_p) \\
&= \sum_{\substack{j_{1\bullet}=l_1 \\ \cdots \\ j_{p\bullet}=l_p}} \prod_{l=1}^p P(X_l=j_l) \cdot \prod_{1 \le l < m \le p} P(X_{lm}=j_{lm}) \cdots P(X_{1 \cdots p}=j_{1 \cdots p}) \\
&= \sum_{\substack{j_{1\bullet}=l_1 \\ \cdots \\ j_{p\bullet}=l_p}} \prod_{l=1}^p \frac{e^{-c_l} c_l^{j_l}}{j_l!} \cdot \prod_{1 \le l < m \le p} \frac{e^{-c_{lm}} c_{lm}^{j_{lm}}}{ j_{lm}! } \cdots \frac{ e^{-c_{1 \cdots p}} c_{1 \cdots p}^{j_{1 \cdots p}} }{ j_{1 \cdots 3}! } \nonumber \\
&\equiv P(Y_1=l_1,\ldots,Y_p=l_p).
\end{align*}

The above remark leads to the simulation algorithm called multivariate reduction scheme.

For the simulation studies, the dimension, $p$, is chosen to take the values  2, 5, 10, and 20. The sample size, $n$, is chosen to vary from values greater than $p$ to 500. For each non-empty $T \in \mathcal{J}$, we simulate independent $c_T$ from a gamma distribution with shape parameter, $a_0=1$ for all dimensions, and rate parameter, $a_1=1, 0.1, 0.01, 0.0001$ for dimensions 2, 5, 10, and 20, respectively. We simulate independent univariate Poisson-distributed random variables $X_T$ with parameter $c_T$. We generate a sample vector $Y$ from the multivariate Poisson distribution  with parameter set $\{c_T: T \in \mathcal{J}\}$  by setting $Y_{j} = \sum_{T \in \mathcal{J} : \, j \in T} X_T$, $j=1,\ldots,p$. After we obtain the sample $Y_1, \ldots, Y_n$, we calculate the mean estimators, as well as their corresponding risk. In order to obtain accurate estimates for the risks, we repeat the above process 200 times and calculate the average risk for each estimator.

\medskip

\begin{figure}[!ht]
\centering
\includegraphics[width=\textwidth, height=10.2cm]{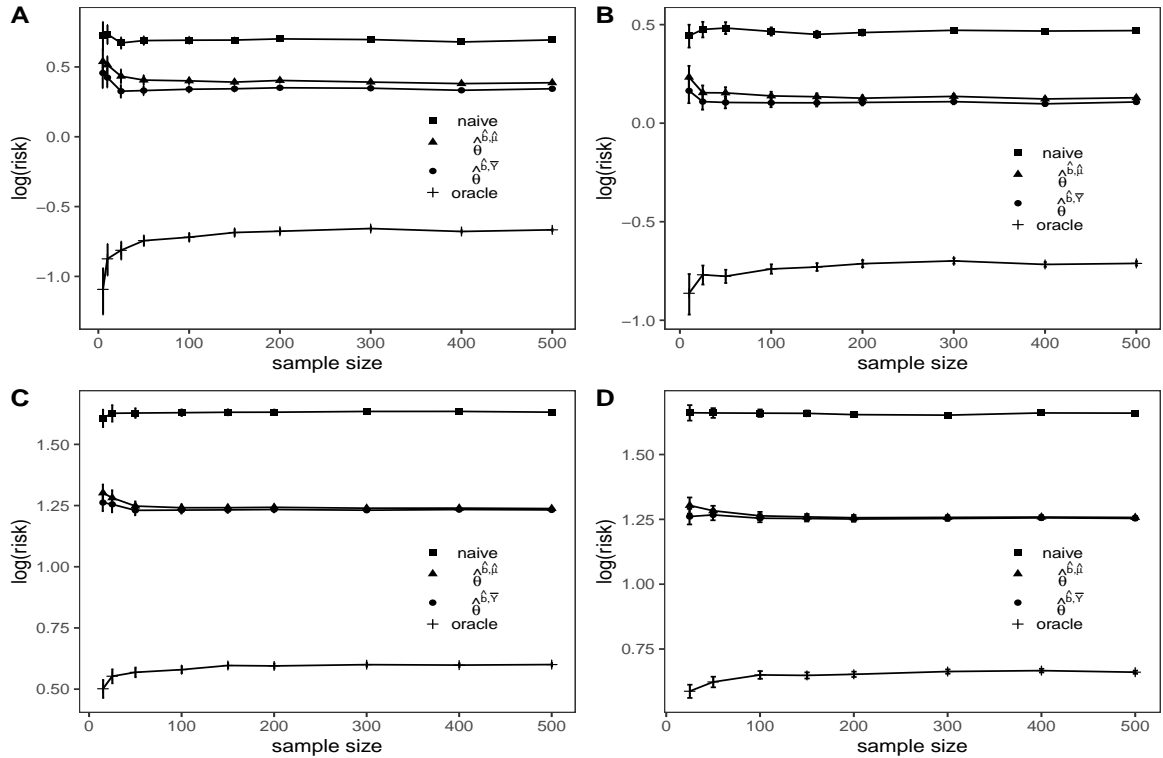}
\caption{Comparisons of the average risk of different estimators for Poisson distribution using the reduction scheme. \textbf{A}: $p=2$. \textbf{B}: $p=5$. \textbf{C}: $p=10$. \textbf{D}: $p=20$.}
\label{fig:poisson_fs_small}
\end{figure}

%\medskip

In Figure \ref{fig:poisson_fs_small}, panels A-D show the logarithm of the average risk of the estimators with 95\% confidence interval for dimension 2, 5, 10, and 20. We again plot the logarithm of the risks to visualize more accurately the size of the risks. The logarithm of the risk of the naive estimator $\hat{\theta}_{ij}$, which is relatively constant for all $n$ as expected, is larger than the other three estimators. The performance of the two shrinkage estimators $\hat{\theta}_{ij}^{\hat{b}^*,\bar{Y}}$ and $\hat{\theta}_{ij}^{\hat{b}^*,\hat{\mu}^*}$ is very good as they are close to the oracle estimators. The shrinkage estimator $\hat{\theta}_{ij}^{\hat{b}^*,\bar{Y}}$ performs better than $\hat{\theta}_{ij}^{\hat{b}^*,\hat{\mu}^*}$ for lower dimensions. However, as we increase the dimension, we observe that the performances of these two estimators are approaching each other.

%-------------------------------------------------------------------------------------------------------

%\clearpage

\phantom{a}

\medskip
%\bigskip

\bibliographystyle{apalike}
\bibliography{bibliography}

\newpage 

\begin{center}
{\Large \textbf{Appendix}}
\end{center}

%-------------------------------------------------------------------------------------------------------

\appendix

\section{Proof of Lemma \ref{lemma_T1_T2_T3}}
\label{appendix_A}

In this section, we establish Lemma \ref{lemma_T1_T2_T3}.  In proving that result, we will apply Doob's $L^r$ maximal inequality \cite[Theorem 3.4, p.~317]{doob1990stochastic}, which we state as follows.  

\begin{lemma}
\label{Doob_max_ineq}
{\rm (Doob's $L^r$ maximal inequality)} Let $\{M_n: n \geq 1\}$ be a martingale. If $r>1$ and $E\big(|M_j|^r \big)< \infty$ for all $0\leq j \leq n$, then
$$
E\Big( \max\limits_{0\leq j \leq n} |M_j|\Big)^r \leq \bigg( \frac{r}{r-1}\bigg)^r E\big(|M_n|^r\big).
$$
\end{lemma}

\smallskip

\begin{proof}[Proof of Lemma \ref{lemma_T1_T2_T3}]  
In bounding $T_1$, we define 
\begin{equation}
\label{Zi_first_exp}
Z_i = \frac{1}{p} \sum_{j=1}^{p} 
\bigg[ \frac{V(Y_{ij})}{\tau_{ij}+\nu_2}-(Y_{ij}-\theta_{ij})^2 \bigg].
\end{equation}
An alternative expression for $Z_i$ is 
\begin{equation}
\label{Zi_second_exp}
Z_i = \frac{1}{p} \sum_{j=1}^{p} \bigg[-\frac{\tau_{ij}}{\tau_{ij}+\nu_2}\big(Y_{ij}^2-E(Y_{ij}^2)\big) + 
\bigg(2\theta_{ij}+\frac{\nu_1}{\tau_{ij} + \nu_2}\bigg)(Y_{ij}-\theta_{ij})\bigg].
\end{equation}
To prove this, we substitute in \eqref{Zi_first_exp} the formula $V(Y_{ij}) = \nu_0 + \nu_1 Y_{ij} + \nu_2 Y_{ij}^2$ from \eqref{V_polynomial}; then the $i$th term in \eqref{Zi_first_exp} is a quadratic polynomial in $Y_{ij}$.  So to prove \eqref{Zi_second_exp}, we need only to verify that the coefficients of $Y_{ij}^k$, $k=0,1,2$ in the $i$th terms in \eqref{Zi_first_exp} and \eqref{Zi_second_exp} are the same.  

For $k = 1, 2$, it is simple to verify that the coefficient of $Y_{ij}^k$ in the $i$th terms in \eqref{Zi_first_exp} and \eqref{Zi_second_exp} are equal.  For $k=0$, i.e., the term which is free of $Y_{ij}$, we need to show that  
\begin{align}
\label{coefficient_k=0}
\frac{\tau_{ij}}{\tau_{ij}+\nu_2} E(Y_{ij}^2) - \bigg(2\theta_{ij}+\frac{\nu_1}{\tau_{ij} + \nu_2}\bigg) \theta_{ij} = \frac{\nu_0}{\tau_{ij}+\nu_2} - \theta_{ij}^2.
\end{align} 
Noting that $\theta_{ij}^2 \equiv (E Y_{ij})^2 = E (Y_{ij}^2) - \Var(Y_{ij})$ and $\tau_{ij}Var(Y_{ij}) = \nu_0 + \nu_1 \theta_{ij} + \nu_2 \theta_{ij}^2$, we find that the left-hand side of \eqref{coefficient_k=0} equals 
\begin{align*}
\frac{\tau_{ij}}{\tau_{ij}+\nu_2}\big(\Var(Y_{ij})+\theta_{ij}^2\big) & - 2\theta_{ij}^2-\frac{\nu_1\theta_{ij}}{\tau_{ij}+\nu_2} \\
&= \frac{\tau_{ij}\Var(Y_{ij}) + \tau_{ij}\theta_{ij}^2}{\tau_{ij}+\nu_2} - 2\theta_{ij}^2-\frac{\nu_1\theta_{ij}}{\tau_{ij}+\nu_2} \\
&=\frac{\nu_0 + \nu_1\theta_{ij} + \nu_2\theta_{ij}^2 + \tau_{ij}\theta_{ij}^2 - 2 (\tau_{ij}+\nu_2)\theta_{ij}^2 - \nu_1\theta_{ij}}{\tau_{ij}+\nu_2} \\
&=\frac{\nu_0 - (\tau_{ij}+\nu_2)\theta_{ij}^2}{\tau_{ij}+\nu_2},
\end{align*}
which equals the right-hand side of \eqref{coefficient_k=0}.

Returning to \eqref{Zi_second_exp}, noting that $E(Z_i)$=0, and applying Jensen's inequality, we obtain 
\begin{align}
\label{lemma_proof_T1_1}
E(Z_i^2) 
&\leq \frac{1}{p}\sum_{j=1}^{p} E\bigg(-\frac{\tau_{ij}}{\tau_{ij}+\nu_2}\big(Y_{ij}^2-E(Y_{ij}^2)\big) + 
\bigg(2\theta_{ij}+\frac{\nu_1}{\tau_{ij} + \nu_2}\bigg)(Y_{ij}-\theta_{ij})\bigg)^2. 
\end{align}
Applying the inequality $(a+b)^2\leq 2(a^2 + b^2)$ for all $a, b \in \mathbb{R}$, we have
\begin{align}
\label{lemma_proof_T1_2}
E(Z_i^2)&\leq \frac{2}{p}\sum_{j=1}^{p} \bigg[ \bigg(\frac{\tau_{ij}}{\tau_{ij}+\nu_2}\bigg)^2 \Var(Y_{ij}^2) + 
\bigg(2\theta_{ij}+\frac{\nu_1}{\tau_{ij} + \nu_2}\bigg)^2 \Var(Y_{ij})\bigg].
\end{align}
By applying the same inequality as in \eqref{lemma_proof_T1_1}, we find that \eqref{lemma_proof_T1_2} is bounded above by
\begin{align*}
\frac{2}{p}&\sum_{j=1}^{p} \bigg[ \bigg(\frac{\tau_{ij}}{\tau_{ij}+\nu_2}\bigg)^2 \Var(Y_{ij}^2) + 
8\theta_{ij}^2 \Var(Y_{ij})+2\bigg(\frac{\nu_1}{\tau_{ij} + \nu_2}\bigg)^2 \Var(Y_{ij})\bigg]\\
\equiv & \frac{2}{p}\sum_{j=1}^{p} \bigg(\frac{\tau_{ij}}{\tau_{ij}+\nu_2}\bigg)^2 \Var(Y_{ij}^2) + 
\frac{16}{p}\sum_{j=1}^{p} \theta_{ij}^2\Var(Y_{ij}) {\hskip-0.3033pt} +
\frac{4}{p}\sum_{j=1}^{p}\bigg(\frac{\nu_1}{\tau_{ij} + \nu_2}\bigg)^2 \Var(Y_{ij}).
\end{align*}
Therefore,
\begin{align*}
 E\bigg(\frac{1}{n} \sum_{i=1}^{n} Z_i\bigg)^2
&= \frac{1}{n^2} \sum_{i=1}^{n} \Var(Z_i) \\ &\leq  \frac{2}{n^2p}\sum_{i=1}^{n}\sum_{j=1}^{p} \bigg(\frac{\tau_{ij}}{\tau_{ij}+\nu_2}\bigg)^2 \Var(Y_{ij}^2)+ \frac{16}{n^2p}\sum_{i=1}^{n}\sum_{j=1}^{p} \theta_{ij}^2\Var(Y_{ij})\\
&  \qquad +\frac{4}{n^2p}\sum_{i=1}^{n}\sum_{j=1}^{p}\bigg(\frac{\nu_1}{\tau_{ij}+\nu_2}\bigg)^2 \Var(Y_{ij}).
\end{align*}
Since $\tau_{ij} \geq 1$ for all $i,j$, we have 
\begin{align*}
\frac{1}{(\tau_{ij}+\nu_2)^2} \leq \frac{\tau_{ij}^2}{(\tau_{ij}+\nu_2)^2},
\end{align*}
and by taking the supremum over $i,j$, we further obtain that
\begin{align*}
\sup\limits_{i,j}\bigg(\frac{1}{\tau_{ij}+\nu_2} \bigg)^2 \leq \sup\limits_{i,j}\bigg(\frac{\tau_{ij}}{\tau_{ij}+\nu_2}\bigg)^2.
\end{align*}
Hence, condition \ref{D} implies that $\sup\limits_{i,j}\big(\nu_1/(\tau_{ij}+\nu_2)\big)^2<\infty$.
Since,
$$T_1=  \bigg| \frac{1}{n}\sum_{i=1}^{n} \frac{1}{p}\sum_{j=1}^{p} 
\bigg[ \frac{V(Y_{ij})}{\tau_{ij}+v_2}-(Y_{ij}-\theta_{ij})^2 \bigg] \bigg|=\bigg| \frac{1}{n}\sum_{i=1}^{n} Z_i \bigg|, $$
and under conditions \ref{A} - \ref{D}, as $n, p \to \infty$, $E(T_1^2)=E\big( \sum_{i=1}^{n} Z_i /n \big)^2 =\text{O}(1/ n)$. Hence, we obtain the desired result.

In bounding $T_2$, we assume, without loss of generality, that $ \tau_{1 \boldsymbol{\cdot}} \leq \cdots \leq \tau_{n \boldsymbol{\cdot}}$ and thus $b_1 \geq \cdots \geq b_n$. Therefore, we obtain
\begin{align*}
 T_2 &= \sup\limits_{1\geq b_1\geq \cdots \geq b_n \geq 0}  \frac{2}{np} \bigg| \sum_{i=1}^{n} \sum_{j=1}^{p} 
b_i \bigg[ \frac{V(Y_{ij})}{\tau_{ij}+v_2}-(Y_{ij}-\theta_{ij})^2 \bigg] \bigg|\\
&=\max\limits_{1 \leq k \leq n}  \frac{2}{np} \bigg| \sum_{i=1}^{k} \sum_{j=1}^{p}  \bigg[ \frac{V(Y_{ij})}{\tau_{ij}+v_2}-(Y_{ij}-\theta_{ij})^2 \bigg] \bigg|
=\max\limits_{1 \leq k \leq n}  \frac{2}{n} \bigg| \sum_{i=1}^{k} Z_i \bigg|.
\end{align*}
Let $M_k= \sum_{i=1}^{k} Z_i$. Then, 
\begin{equation}
\label{M_martingale1}
\begin{aligned}
E(M_{k+1}|M_1,\ldots ,M_k) &= E(Z_1+\cdots+Z_{k+1}| Z_1,\ldots, Z_k) \\
&= E(Z_1| Z_1,\ldots, Z_k) + \cdots + E(Z_{k+1}| Z_1,\ldots, Z_k) .
\end{aligned}
\end{equation}
Using the fact that $Z_i$ are independent and $E(Z_i)= 0$ for all $i=1,\ldots,n$, we obtain 
\begin{align}
\label{M_martingale2}
E(M_{k+1}|M_1,\ldots ,M_k)&= Z_1+\cdots +Z_k+E(Z_{k+1}) \nonumber \\
&=Z_1+ \cdots +Z_k = M_{k}.
\end{align}
Therefore, $\{M_k: k \geq 1\}$ forms a martingale. Applying Lemma \ref{Doob_max_ineq} for $r=2$, we have
\begin{align*}
E\Big(\max\limits_{1 \leq k \leq n} M_k^2\Big) \leq 4 E(M_n^2) &= 4E\bigg( \sum_{i=1}^{n} Z_i\bigg)^2 \\
&= 4\Var\bigg(\sum_{i=1}^{n}Z_i\bigg) = 4\sum_{i=1}^{n} \Var(Z_i) = 4 \sum_{i=1}^{n} E(Z_i^2),
\end{align*}
and thus again we obtain, as $n, p \rightarrow\infty$,  $E(T_2^2) \leq 8 E\big(\sum_{i=1}^{n} Z_i/n\big)^2 = \text{O}(1/n)$. Hence, we obtain the desired result.

In bounding $T_3$, we note that 
\begin{multline}
\label{uniform_URE_proof_0}
\frac{2}{np} \sum_{i=1}^{n} \sum_{j=1}^{p} b_i(Y_{ij}-\theta_{ij})(\theta_{ij}-\mu_{j}) \\
= \frac{2}{np} \sum_{i=1}^{n} \sum_{j=1}^{p} b_i\theta_{ij}(Y_{ij}-\theta_{j}) 
- \frac{2}{np} \sum_{i=1}^{n} \sum_{j=1}^{p} b_i\mu_{j}(Y_{ij}-\theta_{ij}).
\end{multline}
Taking the supremum over $\Lambda$ of the absolute value of \eqref{uniform_URE_proof_0}, we have
that
\begin{align*}
T_3 \leq T_{31}+T_{32}
\end{align*}
where
\begin{align*}
&T_{31}= \sup_{(b,\mu) \in \Lambda}   \frac{2}{np} \bigg| \sum_{i=1}^{n} \sum_{j=1}^{p} 
b_i \theta_{ij}(Y_{ij}-\theta_{ij}) \bigg|,\\
&T_{32}= \sup_{(b,\mu) \in \Lambda} \frac{2}{np} \bigg| \sum_{i=1}^{n} \sum_{j=1}^{p} b_i\mu_{j}(Y_{ij}-\theta_{ij}) \bigg|.
\end{align*}

For the term  $T_{31}$, since $b_1\geq \cdots \geq b_n$, 
\begin{align*}
T_{31}&= \sup\limits_{1\geq b_1\geq \cdots \geq b_n \geq 0}  \frac{2}{np} \bigg| \sum_{i=1}^{n} \sum_{j=1}^{p} 
b_i \theta_{ij}(Y_{ij}-\theta_{ij}) \bigg|\nonumber\\
&=\max\limits_{1 \leq k \leq n}  \frac{2}{np} \bigg| \sum_{i=1}^{k} \sum_{j=1}^{p}  \theta_{ij}(Y_{ij}-\theta_{ij}) \bigg|.
\end{align*}
Let $N_k= \sum_{i=1}^{k} U_i$,  where $U_i=  \sum_{j=1}^{p}  \theta_{ij}(Y_{ij}-\theta_{ij})/p$. By conditioning on $U_1, \ldots, U_k$ and applying the same arguments as in eq. \eqref{M_martingale1} and \eqref{M_martingale2}, we obtain that 
\begin{align*}
E(N_{k+1}|N_1,\ldots ,N_k)&= E(U_1+\cdots+U_{k+1}| U_1,\ldots, U_k) =N_k,
\end{align*}
and therefore $\{N_k: k \geq 1\}$ forms a martingale. Applying Lemma \ref{Doob_max_ineq} for $r=2$, we obtain
\begin{align*}
E\Big(\max\limits_{1 \leq k \leq n} N_k^2\Big) \leq 4 E(N_n^2)=4 \sum_{i=1}^{n} E(U_i^2),
\end{align*}
and under condition \ref{B},  we obtain, as $n, p \rightarrow\infty$,  
$$E(T_{31}^2) \leq \frac{8}{n} E\bigg(\sum_{i=1}^{n} U_i\bigg)^2 \leq \frac{8}{n^2p}\sum_{i=1}^n\sum_{j=1}^p \Var({Y_{ij}})\theta_{ij}^2= \text{O}(1/n).$$ 

For the term $T_{32}$, we have 
\begin{align}
\label{lemma_proof_T32_1}
\frac{2}{np} E \bigg( \sup_{(b,\mu) \in \Lambda} & \bigg| \sum_{i=1}^{n} \sum_{j=1}^{p} b_i \mu_{j}(Y_{ij}-\theta_{ij}) \bigg| \bigg) \nonumber \\
&\leq \frac{2}{np} E \bigg( \sup_{(b,\mu) \in \Lambda}   \sum_{j=1}^{p}\big|\mu_{j}\big|\bigg| \sum_{i=1}^{n} b_i(Y_{ij}-\theta_{ij}) \bigg| \bigg) \nonumber \\
&\leq \frac{2}{np} E \bigg(  \sum_{j=1}^{p}\sup_{(b,\mu) \in \Lambda}\big|\mu_{j}\big|\sup_{(b,\mu) \in \Lambda}\bigg| \sum_{i=1}^{n} b_i(Y_{ij}-\theta_{ij}) \bigg| \bigg).
\end{align}
Since $|\mu_j|\leq \max\{|Y_{il}|: i= 1,\ldots,n, l=1,\ldots,p\} \text{ for all $j=1,\ldots, p$}$, it follows that \eqref{lemma_proof_T32_1} equals to
\begin{align}
\label{lemma_proof_T32_2}
\frac{2}{np} \sum_{j=1}^{p} E \bigg( \max\limits_{i,j}\big|Y_{ij}\big|\sup_{(b,\mu) \in \Lambda}\bigg| \sum_{i=1}^{n} b_i(Y_{ij}-\theta_{ij}) \bigg| \bigg).
\end{align}
Applying the Cauchy-Schwarz inequality, we derive that \eqref{lemma_proof_T32_2} is bounded above by
\begin{multline*}
\frac{2}{np} \sum_{j=1}^{p} \bigg[ E \big( \max\limits_{i,j}Y_{ij}^2 \big) E \bigg( \sup_{(b,\mu) \in \Lambda}\bigg| \sum_{i=1}^{n} b_i(Y_{ij}-\theta_{ij}) \bigg|^2 \bigg) \bigg]^{1/2} \\
\equiv \frac{2}{np} \big(E \big(\max\limits_{i,j}Y_{ij}^2 \big) \big)^{1/2} \sum_{j=1}^{p} \bigg[  E \bigg( \sup_{(b,\mu) \in \Lambda}\bigg| \sum_{i=1}^{n} b_i(Y_{ij}-\theta_{ij}) \bigg|^2 \bigg) \bigg]^{1/2}.
\end{multline*}
For $b_1\geq \cdots \geq b_n$, we apply Lemma \ref{Doob_max_ineq} as before to obtain 
\begin{align}
\label{discuss}
\frac{2}{np} E & \bigg( \sup_{(b,\mu) \in \Lambda} \bigg| \sum_{i=1}^{n} \sum_{j=1}^{p} b_i\mu_{j}(Y_{ij}-\theta_{ij}) \bigg| \bigg) \nonumber \\
&\leq \frac{2}{np} \big( E \big( \max\limits_{i,j}Y_{ij}^2 \big)\big)^{1/2}\sum_{j=1}^{p} \bigg[  E \bigg( \max\limits_{1\leq k \leq n}\bigg| \sum_{i=1}^{k} (Y_{ij}-\theta_{ij}) \bigg|^2 \bigg) \bigg]^{1/2} \nonumber \\
&\leq \frac{2}{np} \big( E \big( \max\limits_{i,j}Y_{ij}^2 \big) \big)^{1/2} \sum_{j=1}^{p}   \bigg[ \sum_{i=1}^{n} \Var(Y_{ij}) \bigg]^{1/2}.
\end{align}
By again applying Jensen's inequality, we find that \eqref{discuss} is bounded above by 
\begin{multline*}
\frac{2}{np} \Big( E \big( \max\limits_{i,j}Y_{ij}^2 \big)\Big)^{1/2} \bigg( p \sum_{i=1}^{n}   \sum_{j=1}^{p} \Var(Y_{ij}) \bigg)^{1/2} \\
\equiv 2 \bigg( \frac{1}{n}E \big( \max\limits_{i,j}Y_{ij}^2 \big)\bigg)^{1/2} \bigg( \frac{1}{np}\sum_{j=1}^{p}    \sum_{i=1}^{n} \Var(Y_{ij}) \bigg)^{1/2};
\end{multline*}
hence, 
\begin{multline*}
\frac{2}{np} E \bigg( \sup_{(b,\mu) \in \Lambda} \bigg| \sum_{i=1}^{n} \sum_{j=1}^{p} b_i\mu_{j}(Y_{ij}-\theta_{ij}) \bigg| \bigg) \\
\le 2\bigg( \frac{1}{n}E \big( \max\limits_{i,j}Y_{ij}^2 \big)\bigg)^{1/2} \bigg( \frac{1}{np}\sum_{j=1}^{p}    \sum_{i=1}^{n} \Var(Y_{ij}) \bigg)^{1/2}.
\end{multline*}
\iffalse
\begin{align*}
\frac{2}{np} E \bigg( \sup_{(b,\mu) \in \Lambda}  \bigg| \sum_{i=1}^{n} \sum_{j=1}^{p} & b_i\mu_{j}(Y_{ij}-\theta_{ij}) \bigg| \bigg) \\
&\leq \frac{2}{np} \big( E \big( \max\limits_{i,j}Y_{ij}^2 \big))^{1/2} \bigg( p \sum_{j=1}^{p}   \sum_{i=1}^{n} \Var(Y_{ij}) \bigg)^{1/2}\\
&\leq \frac{2}{n^{1/2}} \big( E \big( \max\limits_{i,j}Y_{ij}^2 \big)\big)^{1/2} \bigg( \frac{1}{np}\sum_{j=1}^{p}    \sum_{i=1}^{n} \Var(Y_{ij}) \bigg)^{1/2}\\
&\leq 2 \bigg( \frac{1}{n}E \big( \max\limits_{i,j}Y_{ij}^2 \big) \bigg)^{1/2} \bigg( \frac{1}{np}\sum_{j=1}^{p}    \sum_{i=1}^{n} \Var(Y_{ij}) \bigg)^{1/2}\\
&=2n^{-\epsilon/2} \bigg( \frac{1}{n^{1-\epsilon}}E \big( \max\limits_{i,j}Y_{ij}^2 \big)\bigg)^{1/2} \bigg( \frac{1}{np}\sum_{j=1}^{p}    \sum_{i=1}^{n} \Var(Y_{ij}) \bigg)^{1/2}.
\end{align*}
\fi
Therefore, under conditions \ref{A} and \ref{E}, it follows that, as $n, p\rightarrow\infty$, $E(|T_{32}|)=\text{O}(n^{-\alpha/2}p^{\beta/2})$.
\end{proof}

\section{Proof of Proposition \ref{proposition_pdf}}
\label{appendix_B}

In this section, we establish Proposition \ref{proposition_pdf}.

\begin{proof}[Proof of Proposition \ref{proposition_pdf}]
First, we verify that the function in \eqref{mgd_pdf} is a density function.  It is well-known that the generalized hypergeometric series, ${}_0F_{p-1}(\lambda,\ldots,\lambda;y)$, converges for all $y in \mathbb{R}$; see \cite[pp.~62, Theorem 2.1.1]{andrews1999special}.  It is also clear that $f(y) > 0$ for $\lambda > 0$ and all $\beta_j > 0$, $j=1,\ldots,p$.  We will see later that the Laplace transform of $f(y)$ exists in a neighborhood of the origin and equals $1$ when $s = 0$, so the total integral of $f(y)$ is 1.  Therefore, $f(y)$ is a density function.  

The proof that the distribution of $Y$ has the required Laplace transform is obtained by straightforward term-by-term integration.  We obtain 
\begin{align}
\label{lt_gamma_0}
\E\bigg[\exp\bigg(&-\sum_{j=1}^p s_j Y_j\bigg)\bigg] \nonumber \\
&= \frac{C_0(\lambda;\beta)}{[\Gamma(\lambda)]^p} \, \int_{\bR_+^p} {}_0F_{p-1}(\lambda,\ldots,\lambda;y_1\cdots y_p) \prod_{j=1}^p y_j^{\lambda-1} e^{-(s_j+\beta_j) y_j} \dd y_j \nonumber \\
&= \frac{C_0(\lambda;\beta)}{[\Gamma(\lambda)]^p} \, \sum_{k=0}^\infty \frac{1}{k! \, [(\lambda)_k]^{p-1}} \prod_{j=1}^p \int_0^\infty y_j^{\lambda+k-1} e^{-(s_j+\beta_j) y_j} \dd y_j \nonumber \\
&= \frac{C_0(\lambda;\beta)}{[\Gamma(\lambda)]^p} \, \sum_{k=0}^\infty \frac{1}{k! \, [(\lambda)_k]^{p-1}} \prod_{j=1}^p \big[\Gamma(\lambda+k) \, (s_j+\beta_j)^{-(\lambda+k)}\big].
\end{align}
Writing $\Gamma(\lambda+k) = (\lambda)_k \Gamma(\lambda)$, we find that the Laplace transform \eqref{lt_gamma_0} equals
\begin{align}
\label{lt_gamma_1}
\E\bigg[\exp\bigg(-\sum_{j=1}^p s_j Y_j\bigg)\bigg] &= C_0(\lambda;\beta) \, \sum_{k=0}^\infty \frac{(\lambda)_k}{k!} \prod_{j=1}^p (s_j+\beta_j)^{-(\lambda+k)} \nonumber \\
&= C_0(\lambda;\beta) \, \prod_{j=1}^p (s_j+\beta_j)^{-\lambda} \cdot \sum_{k=0}^\infty \frac{(\lambda)_k}{k!} \prod_{j=1}^p (s_j+\beta_j)^{-k}.
\end{align}
With convergence for $|\prod_{j=1}^p (s_j+\beta_j)| > 1$, \eqref{lt_gamma_1} becomes
\begin{align}
\label{lt_gamma_2}
\E\bigg[\exp\bigg(-\sum_{j=1}^p s_j Y_j\bigg)\bigg] &= C_0(\lambda;\beta) \, \prod_{j=1}^p \big[(s_j+\beta_j)^{-\lambda} \cdot {}_1F_0\bigg(\lambda;\prod_{j=1}^p (s_j+\beta_j)^{-1}\bigg) \nonumber \\
&= C_0(\lambda;\beta) \, \prod_{j=1}^p (s_j+\beta_j)^{-\lambda} \cdot \bigg(1 - \prod_{j=1}^p (s_j+\beta_j)^{-1}\bigg)^{-\lambda} \nonumber \\
&= C_0(\lambda;\beta) \, \bigg(\prod_{j=1}^p (s_j+\beta_j) - 1\bigg)^{-\lambda}.
\end{align}
Setting $s = 0$ in \eqref{lt_gamma_2}, we obtain 
$$
C_0(\lambda;\beta) \, (\beta_\star - 1)^{-\lambda} = 1,
$$
which verifies \eqref{C0_constant}.  

Substituting into \eqref{lt_gamma_2} the formula for $C_0(\lambda;\beta)$, we obtain 
\begin{align}
\label{lt_gamma_3}
\E\bigg[\exp\bigg(-\sum_{j=1}^p s_j Y_j\bigg)\bigg] &= (\beta_\star - 1)^\lambda \cdot \bigg(\prod_{j=1}^p (s_j+\beta_j) - 1\bigg)^{-\lambda} \nonumber \\
&= \bigg(\frac{\prod_{j=1}^p (s_j+\beta_j) - 1}{\beta_\star - 1}\bigg)^{-\lambda}.
\end{align}
For each $T \in \mathcal{J}$, let $\bar{T} = \{1,\ldots,p\} \setminus T$ denote the complement of $T$.  Then we have 
\begin{align*}
\prod_{j=1}^p (s_j+\beta_j) &= \sum_{T \in \mathcal{J}} s^T \beta^{\bar{T}} \\
&= s^\emptyset \beta^{\bar{\emptyset}} + \sum_{T \in \mathcal{J}, T \neq \emptyset} s^T \beta^{\bar{T}} 
= \prod_{j=1}^p \beta_j + \sum_{T \in \mathcal{J}, T \neq \emptyset} s^T \beta^{\bar{T}}.
\end{align*}
Therefore 
$$
\prod_{j=1}^p (s_j+\beta_j) - 1 = \beta_\star - 1 + \sum_{T \in \mathcal{J}, T \neq \emptyset} s^T \beta^{\bar{T}},
$$
and it follows that \eqref{lt_gamma_3} equals
\begin{align*}
\E\bigg[\exp\bigg(-\sum_{j=1}^p s_j Y_j\bigg)\bigg] &= \bigg(\frac{\beta_\star - 1 + \sum_{T \in \mathcal{J}, T \neq \emptyset} s^T \beta^{\bar{T}}}{\beta_\star - 1}\bigg)^{-\lambda} \\
&= \bigg(1 + \sum_{T \in \mathcal{J}, T \neq \emptyset} \frac{\beta^{\bar{T}}}{\beta_\star - 1} s^T\big)^{-\lambda}.
\end{align*}
Now define the coefficient mapping, $c: \mathcal{J} \to \bR$, as in \eqref{coeff_function}.  Then we obtain 
$$
\E\bigg[\exp\bigg(-\sum_{j=1}^p s_j Y_j\bigg)\bigg] = \bigg(\sum_{T \in \mathcal{J}} c_T s^T\bigg)^{-\lambda},
$$
proving that the Laplace transform of $Y$ is of the form \eqref{lt_gamma}.
\end{proof}

\section{Proof of Proposition \ref{proposition_pmf_poisson}}
\label{appendix_C}

In this section, we provide a

\begin{proof}[Proof of Proposition \ref{proposition_pmf_poisson}]
An affine polynomial on $\bR^p$ is of the form 
\begin{align*}
P(s_1, \ldots, s_p) &= c_0 + \sum_{j=1}^p c_j s_j + \sum_{1 \le j < k \le p} c_{jk} s_j s_k + \ldots + c_{1\cdots p} s_1 \cdot \ldots \cdot s_p \\
&= \sum_{T \subseteq \mathcal{J}} c_T s^T.
\end{align*}
Let $Y = (Y_1, \ldots ,Y_p)$ be a random vector with m.g.f.
\begin{align}
\label{mgf_mpoisson}
\E(e^{s_1Y_1 + \ldots + s_pY_p}) &= \exp\big(P(e^{s_1}, \ldots ,e^{s_p})\big) \nonumber \\
&= \exp\bigg(c_0 + \sum_{j=1}^p c_j e^{s_j} + \sum_{1 \le j < k \le p} c_{jk} e^{s_j+s_k} + \ldots + c_{1\cdots p} e^{s_1+ \ldots +s_p}\bigg) \nonumber \\
&\equiv \exp\bigg(\sum_{T \subseteq \mathcal{J}} c_T e^{s_T}\bigg).
\end{align}
Evaluating both sides of \eqref{mgf_mpoisson} at $s = 0$, we obtain 
$$
\sum_{T \subseteq \mathcal{J}} c_T = 0;
$$
hence 
\begin{equation}
\label{c_0}
c_\emptyset = - \sum_{T \subseteq \mathcal{J}, T \neq \emptyset} c_T.
\end{equation}
Substituting \eqref{c_0} into the m.g.f. \eqref{mgf_mpoisson}, we obtain 
\begin{equation}
\label{poisson_mgf_pd}
\E(e^{s_1 Y_1 + \ldots + s_p Y_p}) = \exp\bigg(\sum_{T \subseteq \mathcal{J}, T \neq \emptyset} c_T (e^{s_T} - 1)\bigg).
\end{equation}
Set $s_2 = \ldots = s_p = 0$ in \eqref{poisson_mgf_pd} and obtain 
\begin{align*}
\E(e^{s_1 Y_1}) &= \exp\big(c_1 (e^{s_1} - 1) + c_{12} (e^{s_1} - 1) + \ldots + c_{1\cdots p} (e^{s_1} - 1)\big) \\
&= \exp\big(c_{1\bullet} (e^{s_1} - 1)\big),
\end{align*}
where $c_{1\bullet} = c_1 + c_{12} + \ldots + c_{1\cdots p}$; therefore, $Y_1$ is Poisson-distributed with parameter $c_{1\bullet}$.  In general, $Y_j$ is Poisson-distributed with parameter 
$$
c_{j\bullet} = \sum_{T \in \mathcal{J}: \, j \in T} c_T,
$$
$j=1,\ldots,p$.  Replace each $e^{s_j}$ by $s_j$ in \eqref{poisson_mgf_pd} to obtain 
\begin{equation}
\label{mgf_mpoisson_1}
\E(s_1^{Y_1} \cdots s_p^{Y_p}) = \exp\bigg(\sum_{T \subseteq \mathcal{J}, T \neq \emptyset} c_T (s^T - 1)\bigg).
\end{equation}
Since $c_T\neq 0$, \eqref{mgf_mpoisson_1} becomes
\begin{align}
\label{mgf_mpoisson_2}
\E(s_1^{Y_1} \cdots s_p^{Y_p}) &= \prod_{j=1}^p e^{c_j (s_j - 1)} \cdot \prod_{1\le j<k\le p} e^{c_{jk} (s_j s_k - 1)} \cdot \ldots \cdot e^{c_{1\cdots p} (s_1 \cdots s_p - 1)}.
\end{align}
Now apply the expansion, 
$$
e^{c (s - 1)} = \sum_{j=0}^\infty \frac{e^{-c} c^j}{j!} s^j,
$$
to each term in \eqref{mgf_mpoisson_2}.  Then, we obtain 
\begin{align*}
\E(s_1^{Y_1} \cdots s_p^{Y_p}) &= \sum_{j_1,\ldots ,j_p=0}^\infty \cdots  \sum_{j_{1\cdots p}=0}^\infty \bigg(\prod_{l=1}^p \frac{e^{-c_l} c_l^{j_l} s_l^{j_l}}{j_l!}\bigg) \\
& \qquad\qquad \cdot \bigg(\prod_{1 \le l < m \le p} \frac{ e^{-c_{lm}} c_{lm}^{j_{lm}} s_l^{j_{lm}}  s_m^{j_{lm}} }{ j_{lm}! }\bigg)  \cdots \frac{ e^{-c_{1\cdots p}} c_{1 \cdots p}^{j_{1\cdots p}} s_1^{j_{1 \cdots p}} \cdots s_p^{j_{1\cdots p}}}{j_{1 \cdots p}!}.
\end{align*}

For each non-empty $T \in \mathcal{J}$ with $T = \{l_1,\ldots,l_r\}$ with $l_1 < \cdots < l_r$, define $j_T = j_{l_1 \cdots l_r}$. For each $k=1,\ldots,p$, also define 
$$
j_{k\bullet} = \sum_{T \in \mathcal{J} : \, k \in T} j_T.
$$
Then, we obtain the joint density function of $(Y_1, \ldots ,Y_p)$ as 
$$
P(Y_1=l_1, \ldots ,Y_p=l_p) = \sum_{\substack{j_{1\bullet}=l_1 \\ \cdots  \\ j_{p\bullet}=l_p}} \bigg(\prod_{l=1}^p \frac{e^{-c_l} c_l^{j_l}}{j_l!}\bigg) \cdot \bigg(\prod_{1 \le l < m \le p} \frac{e^{-c_{lm}} c_{lm}^{j_{lm}}}{ j_{lm}! }\bigg) \cdots \bigg(\frac{ e^{-c_{1\cdots p}} c_{1 \cdots p}^{j_{1 \cdots p}} }{ j_{1 \cdots p}! }\bigg).
$$
The proof now is complete.  
\end{proof}

\end{document}